\documentclass{amsart}
\usepackage{amsmath,amsfonts,mathrsfs,enumerate,graphicx,subfig,wasysym}
\usepackage{color} 
\newtheorem{theorem}{Theorem}
\newtheorem{question}{Question}
\newtheorem{proposition}{Proposition}[section]
\newtheorem{corollary}[proposition]{Corollary}
\newtheorem{lemma}[proposition]{Lemma}

\newtheorem{definition}[proposition]{Definition}

\DeclareMathOperator{\dist}{dist}

\newcommand{\triple}[1]{{\left\vert\kern-0.25ex\left\vert\kern-0.25ex\left\vert #1 
    \right\vert\kern-0.25ex\right\vert\kern-0.25ex\right\vert}}

\newcommand{\I}{\mathcal{I}}

\newcommand{\threebar}[1]{{\left\vert\kern-0.25ex\left\vert\kern-0.25ex\left\vert #1 
    \right\vert\kern-0.25ex\right\vert\kern-0.25ex\right\vert}}
\newcommand{\J}{\mathcal{J}}

\newcommand{\R}{\mathbb{R}}
\newcommand{\A}{\mathsf{A}}

\newcommand{\D}{\mathsf{D}}
\newcommand{\C}{\mathsf{C}}
\newcommand{\B}{\mathsf{B}}
\newcommand{\N}{\mathbb{N}}

\newcommand{\Z}{\mathbb{Z}}

\title{On marginal growth rates of matrix products}
\author{Jonah Varney and Ian D. Morris}
\address{J. Varney: Mathematics Department, University of Surrey, Guildford GU2 7XH, United Kingdom}
\email{jonahvarney@gmail.com}
\address{I. D. Morris: School of Mathematical Sciences, Queen Mary University of London, Mile End Road, London E1 4NS, United Kingdom}
\email{i.morris@qmul.ac.uk }
\begin{document}

\begin{abstract}
In this article we consider the maximum possible growth rate of sequences of long products of $d \times d$ matrices all of which are drawn from some specified compact set which has been normalised so as to have joint spectral radius equal to $1$. We define the \emph{marginal instability rate sequence} associated to such a set to be the sequence of real numbers whose $n^{th}$ entry is the norm of the largest product of length $n$, and study the general properties of sequences of this form. We describe how new marginal instability rate sequences can be constructed from old ones, extend an earlier example of Protasov and Jungers to obtain marginal instability rate sequences whose limit superior rate of growth matches various non-integer powers of $n$, and investigate the relationship between marginal instability rate sequences arising from finite sets of matrices and those arising from sets of matrices with cardinality $2$. We also give the first example of a finite set whose marginal instability rate sequence is asymptotically similar to a polynomial with non-integer exponent. Previous examples had this property only along a subsequence.
\end{abstract}
\maketitle


\section{Introduction and overview}

If $\mathsf{A}$ is a bounded, nonempty set of real $d \times d$ matrices, the \emph{joint spectral radius} of $\mathsf{A}$ is defined to be the limit
\begin{equation}\label{eq:jsr}\varrho(\mathsf{A})=\lim_{n \to\infty} \max_{A_1,\ldots,A_n \in \mathsf{A}} \left\|A_n\cdots A_1\right\|^{\frac{1}{n}}= \inf_{n \geq1} \max_{A_1,\ldots,A_n \in \mathsf{A}} \left\|A_n\cdots A_1\right\|^{\frac{1}{n}}\end{equation}
where the existence of the limit is guaranteed by a subadditivity argument. (Here and throughout this article $\|\cdot\|$ will refer always to the Euclidean norm on vectors or to the corresponding operator norm on matrices, as appropriate.) The joint spectral radius of $\mathsf{A}$ is zero if and only if every product of $d$ elements of $\mathsf{A}$ is zero, and this in turn occurs precisely when the elements of $\mathsf{A}$ may be simultaneously conjugated to a set of matrices all of which are upper triangular with zero diagonal (see \cite[\S2.3]{Ju09}). In the remainder of this article we will concern ourselves exclusively with the situation in which $\varrho(\A) > 0$.

The joint spectral radius arises naturally in a number of contexts including the stability theory of linear inclusions in discrete time \cite{Ba88,Gu95}, the regularity of wavelets, fractals and solutions to refinement equations \cite{DaLa92,KoKuCh19,Pr06} and certain combinatorial enumeration problems \cite{BlCaJu09,DuSiTh99}. While its study ultimately dates back to the 1960 article \cite{RoSt60} (reprinted in \cite{Ro03}) it has been the subject of intensive and systematic research principally in the last three decades, and continues to be an active topic of research (see e.g. \cite{BrSe21,Pr21,WaMaMa21}).

In this article we will be concerned with the growth of sequences of the form
\begin{equation}\label{eq:that-sequence}n \mapsto  \max_{A_1,\ldots,A_n \in \mathsf{A}} \left\|A_n\cdots A_1\right\|\end{equation}
at a finer scale of detail than the exponential growth rate considered in \eqref{eq:jsr}. By consideration of the definition \eqref{eq:jsr} it is clear that any sequence of the form \eqref{eq:that-sequence} may be written in the form $\varrho(\A)^n a_n$ where $(a_n)$ satisfies $\lim_{n \to \infty} a_n^{1/n}=1$, and our interest will be in the structure of the sequences $(a_n)$ which arise in this way. It is easily seen that the sequence $(a_n)$ associated to the set $\A$ in this way is left unchanged if $\A$ is replaced with the set $\{ \lambda^{-1}A \colon A \in \A\}$ for some $\lambda>0$, so by replacing $\A$ with $\{ \varrho(\A)^{-1}A \colon A \in \A\}$ where necessary we will without loss of generality specialise our investigation to the study of sequences of the form \eqref{eq:that-sequence} for sets of matrices $\A$ which satisfy $\varrho(\A)=1$. We will refer to such sequences $(a_n)$ as \emph{marginal instability rate sequences}.

This choice of terminology is motivated by the relevance of these sequences to \emph{linear switching systems} which motivates much previous research such as \cite{BlTs00,ChMaSi12,Mo22,PrJu15,Su08}. If $\A$ is a set of real $d \times d$ matrices, the discrete-time linear switching system associated to $\A$ is the set of all sequences $(v_n)_{n=0}^\infty$ which satisfy $v_{n+1}=A_n v_n$ for every $n \geq 1$, for some sequence $(A_n)_{n=1}^\infty \in \A^{\N}$ (called a \emph{switching law}) and initial vector $v_0 \in \R^d$. Linear switching systems have been extensively studied in the control theory literature as models of a dynamical system which admits several ``states'' or behaviours $A \in \A$, and which evolves according to whichever state is chosen by a controller to apply at each time $n \geq 1$. In the case where $\varrho(\A)<1$ it is immediate that all trajectories of the linear switching system defined by $\A$ converge exponentially to the origin and that the rate of convergence is uniform with respect to the choice of switching law and of initial vector (provided that the initial vector is chosen to be a unit vector). In this case the system is referred to as being exponentially stable. On the other hand in the case $\varrho(\A)>1$ there always exists a trajectory which escapes to infinity at exponential speed (see e.g. \cite[Corollary 1.2]{Ju09}) and in this case the system is conventionally called exponentially unstable. The case $\varrho(\A)=1$ is subdivided into the case where the trajectories whose initial vector has unit norm are uniformly bounded (usually called the \emph{marginally stable} case) and the case where no such uniform bound exists (usually called the \emph{marginally unstable}) case. If $\A$ satisfies $\varrho(\A)=1$ and defines a marginally unstable linear switching system then the value $a_n$ of the associated sequence $(a_n)$ is precisely the largest possible norm achievable by a trajectory of the linear switching system at time $n$, assuming that the switching law may be chosen arbitrarily and that the initial vector $v_0$ is allowed to be an arbitrary unit vector. This relationship between marginally unstable linear switching systems and the sequences $(a_n)$ discussed in this article has motivated a substantial amount of previous research on the latter, as noted above; on the other hand, research into marginal instability rate sequences has also been motivated by the study of \emph{$k$-regular sequences} in theoretical computer science \cite{BeCoHa14,BeCoHa16} and by a number of other considerations: for a brief discussion see \cite[\S6]{JuPrBl08}.


In previous research it has been observed independently by a number of authors (such as in \cite{Be05,GuZe01,Su08}) that if $(a_n)$ is a marginal instability rate sequence then the asymptotic growth of $(a_n)$ satisfies constraints additional to the trivial property $\lim_{n \to \infty} a_n^{1/n}$: specifically, we necessarily have $C^{-1} \leq  a_n \leq Cn^{d-1}$ where $d$ is the dimension of the matrices used to define the sequence $(a_n)$ and $C>0$ is a real number depending only on $\A$. The bound $O(n^{d-1})$ was subsequently refined in \cite{ChMaSi12} to an integer power $O(n^\ell)$ (where in general we may have $\ell<d-1$) which depends on certain features of simultaneous block triangularisations of the elements of $\A$. Further works such as \cite{BeCoHa16,JuPrBl08,PrJu15} investigated the problem of determining when $C^{-1}n^\alpha \leq a_n  \leq Cn^\alpha$ for some $\alpha \in [0,d]$ and constant $C>0$, and what values this exponent $\alpha$ may take when this equivalence holds, particularly in the case where $\A$ is a finite set. It was recently demonstrated in \cite{Mo22} that we do not universally have $C^{-1}n^\alpha \leq a_n  \leq Cn^\alpha$ for some $\alpha \in [0,d]$ and that the sequence $\log a_n / \log n$ may in fact be divergent.

The overall structure of this article is as follows. In the next section we formally define marginal instability rate sequences, present two fundamental examples and collect together some minor preliminary results which will be of use in later sections. The subsequent sections  then pass through results demonstrating that the class of marginal instability rate sequences (when calculated specifically with respect to the Euclidean operator norm) is closed under the operations of termwise maximum and termwise product; that the marginal instability rate sequence of a \emph{finite} set $\A$ may in some cases be represented as the marginal instability rate sequence of a second finite set $\B$ which has exactly two elements; and that if $(a_n)$ is a marginal instability rate sequence then the sequence of ratios $a_{n+1}/a_n$ cannot accumulate at zero. We also relate certain marginal instability sequences to subadditive sequences, modify a construction of Protasov and Jungers from \cite{PrJu15} so as to construct marginal instability rate sequences satisfying $a_n \leq C n^\gamma$ for all $n \geq 1$ and $n^{\gamma}_\ell \geq C^{-1} a_{n_\ell}$ along certain subsequences $(n_\ell)$, where the exponent $\gamma >0$ may take any prescribed value in a certain range. We also give an example of a pair of matrices whose marginal instability rate sequence $(a_n)$ satisfies $C^{-1}n^{1/3} \leq a_n \leq Cn^{1/3}$ for all $n \geq 1$, strengthening an example given in \cite{PrJu15} which had the weaker property that $a_n \leq C n^{1/3}$ for all $n\geq 1$ and $n^{1/3} \leq C a_n$ for infinitely many $n \geq 1$.
\section{Preliminaries}

\subsection{Fundamental definitions and notation}

Throughout this work $M_d(\R)$ denotes the set of all $d \times d$ real matrices and $M_d(\mathbb{C})$ the set of all $d \times d$ complex matrices. The notation $M_{d_1 \times d_2}(\R)$ denotes the set of all real matrices of dimensions $d_1 \times d_2$. Given a set $\A \subset M_d(\R)$ and $n \geq 1$ we will write $\A_n:=\{A_1\cdots A_n \colon A_i \in \A\}$ for every $n \geq 1$.

\begin{definition}
A sequence of real numbers $(a_n)_{n=1}^\infty$ is called a \emph{marginal instability rate sequence} if there exist an integer $d \geq 1$, a norm $\threebar{\cdot}$ on $M_d(\R)$ and a set $\mathsf{A} \subset M_d(\R)$ satisfying $\varrho(\mathsf{A})=1$ such that $a_n=\max_{A \in \A_n}\threebar{A}$ for every $n \geq 1$. We describe this relation by saying that $(a_n)$ is the marginal instability rate sequence of $\A$ with respect to the norm $\threebar{\cdot}$.
\end{definition}
In general the above definition permits a marginal instability rate sequence to depend on the choice of norm on $M_d(\R)$, and so a single set $\A$ will in general admit multiple different marginal instability rate sequences $(a_n)$. On the other hand since any two norms on a finite-dimensional space are equivalent to one another, if $(a_n)$ and $(a_n')$ are both marginal instability rate sequences for the same set $\A$ then the ratio $a_n/a_n'$ is bounded away from zero and infinity independently of $n$. In general, therefore, comparison between marginal instability rate sequences, or comparison of a marginal instability rate sequence and another sequence, is only meaningful to within multiplication by a uniform scalar constant. To reflect this we introduce the following additional notation. Given sequences of positive real numbers $(a_n)$, $(b_n)$ we use the notation $a_n \apprle b_n$ to mean that there exists $C>0$ such that $a_n \leq Cb_n$ for all $n \geq 1$, and we write $a_n \simeq b_n$ in the case where both $a_n \apprle b_n$ and $a_n \apprge b_n$. If $(n_\ell)_{\ell=1}^\infty$ is a subsequence of the natural numbers then we will also write $(a_{n_\ell}) \apprle (b_{n_\ell})$ to mean that $a_{n_\ell} \leq C b_{n_\ell}$ for all $\ell \geq 1$, for some unspecified positive constant $C$. In many circumstances we will by default calculate marginal instability rate sequences with respect to the Euclidean norm, but in certain arguments the use of an alternative norm will prove advantageous. In this article these alternative norms will most often be operator norms induced by norms on $\R^d$, but there is no necessity for this constraint to hold in general.

\subsection{Two fundamental examples}

The following simple result is occasionally remarked on in works on the present subject such as \cite{GuZe01,Mo22}. Since it may be regarded as a fundamental motivating example we include a proof for completeness.
\begin{proposition}
Let $\A\subset M_d(\R)$ be a singleton set $\A=\{A\}$ such that $\varrho(\A)=1$. Then the marginal instability rate sequence $(a_n)$ of $\A$ satisfies $a_n \simeq n^k$ for some integer $k$ in the range $0 \leq k<d$. 
\end{proposition}
\begin{proof}
By Gelfand's formula it is clear that $\varrho(\A)$ is simply the ordinary spectral radius $\rho(A)$ of the matrix $A$, so it is sufficient to show that if $A \in M_d(\R)$ satisfies $\rho(A)=1$ then $\|A^n\|\simeq n^k$ for some integer $k$ in the range $0\leq k<d$. It will be convenient to work in $M_d(\mathbb{C})$ instead of $M_d(\R)$.

Let $A \in M_d(\mathbb{C})$ with $\rho(A)=1$. Given any invertible matrix $X \in M_d(\mathbb{C})$ it is clear that $\|A^n\|\simeq \|X^{-1}A^nX\|=\|(X^{-1}AX)^n\|$ so by replacing $A$ with a similar matrix we may without loss of generality assume that $A$ is in Jordan normal form. If $A$ is equal to a simple Jordan matrix $J=\lambda I +N$ where $N^{d-1} \neq 0$ and $N^d=0$, then $|\lambda|=\rho(A)=1$ and we have
\[\|A^n\| = \|J^n\| = \left\|\sum_{j=0}^{n} {n \choose j} \lambda^{n-j} N^j \right\| = \left\|\sum_{j=0}^{d-1}{n \choose j} \lambda^{n-j}N^j \right\| \]
for all $n \geq d$. Since ${n \choose j} \simeq n^j$ for each $j\geq 0$ and since $N^{d-1}\neq 0$ it follows easily that $\|A^n\|\simeq {n \choose d-1} \simeq n^{d-1}$ in this case. In the general case let us write $A$ as a direct sum $A=J_1 \oplus J_2 \oplus \cdots \oplus J_m\oplus R$ where each $J_j$ is a simple Jordan matrix of spectral radius $1$ and dimenson $\ell_j$, say, and where $\rho(R)$ is either zero-dimensional or has spectral radius strictly less than $1$. It is clear that $\|A^n\| \simeq \max_{1 \leq j \leq m} \|J_j^n\| \simeq n^{\max_j \ell_j -1}$ and since $1 \leq \max_{1 \leq j\leq m} \ell_j \leq \sum_{j=1}^m \ell_j \leq d$ we are done.
\end{proof}
We will call a set $\A\subset M_d(\R)$ \emph{reducible} if there exists a vector subspace $V \subseteq \R^d$ which satisfies $0 < \dim V <d$ and  $\bigcup_{A \in \A} AV\subseteq V$. If $\A$ is not reducible then we call it \emph{irreducible}. We also call $\A$ \emph{product bounded} if $\sup_{n \geq 1} \sup_{A \in \A_n} \|A\|<\infty$. The following important facts have a long history, being closely related to results of N.E. Barabanov \cite{Ba88} and of Rota and Strang \cite{RoSt60} respectively:
\begin{proposition}\label{pr:irr}
Let $d \geq 1$ and let $\mathsf{A} \subset M_d(\R)$ be compact and nonempty. Then:
\begin{enumerate}[(i)]
\item
If $\A$ is irreducible and satisfies $\varrho(\A)=1$, then it is product bounded.
\item
If $\A$ is product bounded then there exists a norm $\threebar{\cdot}$ on $\R^d$ which satisfies $\max_{A \in \mathsf{A}_n}\threebar{A}\leq 1$ for every $n \geq 1$.
\end{enumerate}
\end{proposition}
\begin{proof}
The statement (i) is a special case of \cite[Theorem 2.1]{Ju09}. To prove (ii) it suffices to define
\[\threebar{v}:=\sup_{n \geq 0} \sup_{A \in \A_n} \|Av\|\]
for all $v \in \R^d$, where $\A_0$ is understood to be the singleton set containing the identity matrix. The hypothesis that $\A$ is product bounded ensures that $\threebar{\cdot} \colon \R^d \to [0,\infty)$ is well-defined. The property $\max_{A \in \A} \threebar{Av} \leq \threebar{v}$ follows immediately from the definition, and it is also straightforward to verify that $\threebar{\cdot}$ has the properties of a norm.
\end{proof}
A fundamental consequence of the above is that interesting marginal instability rate sequences arise exclusively from reducible sets $\A$:
\begin{corollary}\label{co:lon}
Let $(a_n)$ be the marginal instability rate sequence of a bounded nonempty set $\A\subset M_d(\R)$ with respect to some norm on $M_d(\R)$. Then $(a_n)$ is bounded below by a positive constant, and if $\A$ is irreducible then additionally $(a_n)$ is bounded above.
\end{corollary}
\begin{proof}
Let $(a_n')$ denote the marginal instability rate sequence of $\A$ calculated according to the Euclidean norm. It follows directly from \eqref{eq:jsr} that $a_n' \geq 1$ for every $n \geq 1$ since otherwise the property $\varrho(\A)=1$ is contradicted, whence $a_n \simeq a_n' \apprge 1$ as required. If $\A$ is irreducible, let $\threebar{\cdot}$ be the norm on $\R^d$ given by Proposition \ref{pr:irr} and let $(a_n'')$ be the marginal instability rate sequence of $\A$ with respect to $\threebar{\cdot}$, then $a_n \simeq a_n'' \equiv 1$ as required.
\end{proof}

\subsection{Miscellaneous lemmas}

We collect here some minor results which will be useful in later sections.
\begin{lemma}\label{le:irr-e}
If $\A\subset M_d(\R)$ is compact, nonempty and irreducible then there exists $\varepsilon>0$ such that for every $k \geq 1$, for every $B \in M_{d\times k}(\R)$ we have $\max_{A \in \A} \|AB\| \geq \varepsilon \|B\|$.
\end{lemma}
\begin{proof}
Fix such a set $\A$ throughout the proof. We first show that there exists $\varepsilon>0$ such that for every $v \in \R^d$ we have $\max_{A \in \A} \|Av\| \geq \varepsilon \|v\|$. By homogeneity it suffices to show this for all unit vectors $v$, which is to say we must show that
\[ \min_{v \in \R^d \colon \|v\|=1}\max_{A \in \A} \|Av\|>0.\]
If this is false then by continuity and compactness there exists a unit vector $v \in \R^d$ such that $\max_{A \in \A}\|Av\|=0$. In this case the one-dimensional subspace of $\R^d$ spanned by $v$ is preserved by every $A \in \A$, contradicting irreducibility, and we deduce the existence of $\varepsilon>0$ with the aforementioned property. Now given $k \geq 1$ and $B \in M_{d \times k}(\R)$ choose a unit vector $w \in \R^k$ such that $\|Bw\|=\|B\|$. We have
\[\max_{A \in \A} \|AB\| =\max_{A \in \A} \max_{u \in \R^k \colon \|u\|=1} \|ABu\| \geq \max_{A \in \A} \|ABw\| \geq \varepsilon \|Bw\|=\varepsilon\|B\|\]
as required.
\end{proof}
Though elementary, the following result will be repeatedly found useful:
\begin{lemma}\label{le:tricalc}
Let $A \in M_{d_1}(\R)$, $B \in M_{d_2}(\R)$, $D \in M_{d_1\times d_2}(\R)$. Then
\[\max\{\|A\|, \|B\|\}\leq \left\|\begin{pmatrix} A&D\\0&B\end{pmatrix}\right\| \leq \max\{\|A\|+\|D\|, \|B\|+\|D\|\}\]
where $\|D\|$ denotes the Euclidean operator norm of $D$ considered as a linear map $\R^{d_2} \to \R^{d_1}$.
\end{lemma}
\begin{proof}
 For the lower bound we simply note that
\begin{align*}\left\|\begin{pmatrix} A&D\\0&B\end{pmatrix}\right\| &=\max_{\substack{u \in \R^{d_1}, v \in \R^{d_2}\\ \|u\|^2+\|v\|^2=1}} \left\|\begin{pmatrix} A&D\\0&B\end{pmatrix}\begin{pmatrix}u\\ v\end{pmatrix}\right\|   \\
&\geq \max\left\{\max_{\substack{u \in \R^{d_1} \\ \|u\|=1}} \left\|\begin{pmatrix} A&D\\0&B\end{pmatrix}\begin{pmatrix}u\\ 0\end{pmatrix}\right\| , \max_{\substack{v \in \R^{d_2}\\ \|v\|=1}} \left\|\begin{pmatrix} A&D\\0&B\end{pmatrix}\begin{pmatrix}0\\ v\end{pmatrix}\right\| \right\} \\
&= \max\left\{\|A\|,  \max_{\substack{v \in \R^{d_2}\\ \|v\|=1}} \left\|\begin{pmatrix}Dv\\ Bv\end{pmatrix}\right\| \right\}\geq \max\{\|A\|, \|B\|\}\end{align*}
as required. To obtain the upper bound we observe that by the triangle inequality
\[\left\|\begin{pmatrix} A&D\\0&B\end{pmatrix}\right\| \leq \left\|\begin{pmatrix} A&0\\0&B\end{pmatrix}\right\|+\left\|\begin{pmatrix} 0&D\\0&0\end{pmatrix}\right\|=\max\{\|A\|, \|B\|\}+\|D\|.\]
\end{proof}
The following now-standard result may be obtained easily as a corollary of the Berger-Wang formula,
\[\varrho(\A)=\sup_{n \geq 1} \sup_{A \in \A} \rho(A)^{\frac{1}{n}},\]
an identity which may be found in many sources such as \cite{BeWa92,Bo03,Br22,El95,Ju09,Or20}. Strictly speaking the below result more commonly appears as a lemma in the proof of the Berger-Wang formula, as is the case for example in \cite{BeWa92,Ju09}.
\begin{proposition}\label{pr:uti}
Let $\A \subset M_d(\R)$ be bounded and nonempty. Suppose that there exists an integer $k\geq 1$ and invertible matrix $X \in M_d(\R)$ such that we may write
\begin{equation}\label{eq:blocky}A_i = X^{-1} \begin{pmatrix} A^{(1)}_i& D^{(1,2)}_i &  D^{(1,3)}_i& \cdots &D^{(1,k)}_i\\ 
0& A^{(2)}_i& D^{(2,2)}_i &   \cdots &D^{(2,k)}_i\\ 
0& 0& A^{(3)}_i&   \cdots &D^{(3,k)}_i\\ 
\vdots &\vdots&\vdots&\ddots&\vdots\\
0&0&0&\cdots&A_i^{(k)}
\end{pmatrix}X
\end{equation}
for every $A_i \in \A$, where each $A_i^{(j)}$ is a square matrix whose dimension depends only on $j$ and not on $i$. For each $j=1,\ldots,k$ define $\A^{(j)}=\{A^{(j)}_i \colon A_i \in \A\}$. Then $\varrho(\A)=\max_{1 \leq j \leq k} \varrho(\A^{(j)})$. 
\end{proposition}
\begin{proof}
The case $k=2$ is proved in \cite[Proposition 1.5]{Ju09}. The general case follows easily by induction on $k$, partitioning the blocks in \eqref{eq:blocky} into groups to make a $2\times 2$ block matrix and then applying the case $k=2$ as an induction step. Alternatively, note that for every $n \geq 1$ and $i_1,\ldots,i_n \in \I$,
\[\rho(A_{i_1}\cdots A_{i_n}) =\rho(XA_{i_1}\cdots A_{i_n}X^{-1}) = \max_{1 \leq j \leq k} \rho\left(A_{i_1}^{(j)}\cdots A_{i_n}^{(j)}\right)\]
since the characteristic polynomial of a block upper-triangular matrix is simply the product of the characteristic polynomials of the diagonal blocks. Consequently
\begin{align*}\varrho(\A)=\sup_{n \geq 1} \sup_{A \in \A_n} \rho(A)^{\frac{1}{n}} &= \sup_{n \geq 1} \sup_{A \in \A_n^{(j)}}  \max_{1 \leq j \leq k}\rho(A)^{\frac{1}{n}}\\
&= \max_{1 \leq j \leq k} \sup_{n \geq 1} \sup_{A \in \A_n^{(j)}} \rho(A)^{\frac{1}{n}}=\max_{1 \leq j \leq k} \varrho\left(\A^{(j)}\right)\end{align*}
where the first and last equations follow from the Berger-Wang formula.
\end{proof}

%
%

\section{Combining and modifying marginal instability rate sequences}
In this section we investigate the ways in which new examples of marginal instability rate sequences can be constructed from existing ones. During the preparation of this article we became aware that similar results were attained in \cite{GuZe01}, where the sequence $(\max_{1 \leq k \leq n} a_k)_{n=1}^\infty$ is studied. In this section all marginal instability rate sequences will be calculated with respect to the Euclidean norm on the appropriate space. The key results of this section are summarised as:
\begin{theorem}\label{th:first-obe}
Suppose that $(a_n)$ and $(b_n)$ are the marginal instability rate sequences of nonempty compact sets of matrices $\mathsf{A}$ and $\mathsf{B}$, calculated with respect to the Euclidean norm, and let $k \geq 1$. Then $(\max\{a_n,b_n\})$, $(a_n b_n)$ and $(a_n^k)$ are also the marginal instability rates of nonempty compact sets of matrices. If both $\mathsf{A}$ and $\mathsf{B}$ are finite, or consist only of invertible elements, or both, then the sets of matrices for which $(\max\{a_n,b_n\})$, $(a_n b_n)$ and $(a_n^k)$  are marginal instability rate sequences may be chosen to have the same property.
\end{theorem}
This result will be obtained from the combination of three individual results which are given below. Each of these results is somewhat more detailed than Theorem \ref{th:first-obe} in that it makes the underlying construction explicit. We recall that the Kronecker product $A\otimes B$ of the matrices
\[A=\begin{pmatrix}
a_{11}&\cdots &a_{1d_1}\\
\vdots&\ddots&\vdots\\
a_{d_11}&\cdots &a_{d_1d_1}
\end{pmatrix} \in M_{d_1}(\R),\qquad 
B=\begin{pmatrix}
b_{11}&\cdots &a_{1d_2}\\
\vdots&\ddots&\vdots\\
b_{d_21}&\cdots &a_{d_2d_2}
\end{pmatrix} \in M_{d_2}(\R),\]
is defined to be the matrix
\[A\otimes B=\begin{pmatrix} 
a_{11}B &a_{12}B&\cdots &a_{1d_1}B\\
a_{21}B &a_{22}B&\cdots& a_{2d_1}B\\
\vdots&\vdots&\ddots&\vdots\\
a_{d_1}B&a_{d_12}B&\cdots &a_{d_1d_1}B
\end{pmatrix} \in M_{d_1d_2}(\R).
\]
The operation $\otimes$ is associative. For every $A \in M_d(\R)$ and $k \geq 1$ we define $A^{\otimes k} \in M_{kd}(\R)$ to be the matrix $A\otimes \cdots \otimes A$ with $k$ occurrences of $A$. For further details regarding this construction we direct the reader to \cite[\S4.2]{HoJo94}.

Our first result concerns products of marginal instability rate sequences:
\begin{proposition}\label{pr:stick-with-the-prod}
Let $\I$ and $\J$ be nonempty sets,  let $d_1,d_2 \geq 1$, and suppose that $\A=\{A_i \colon i \in \I\}\subset M_{d_1}(\R)$ and $\B=\{B_i \colon i \in \J\}\subset M_{d_2}(\R)$ are compact sets such that $\varrho(\A)=\varrho(\B)=1$. Let $(a_n)$ and $(b_n)$ be  the marginal instability rate sequences of $\A$ and $\B$ respectively, calculated with respect to the Euclidean norm. Define
\[\C:=\left\{A_i \otimes B_j \colon i \in \I\text{ and }j \in \J\right\}.\]
Then $\varrho(\C)=1$ and $c_n=a_n b_n$ for every $n \geq 1$.
\end{proposition}
\begin{proof}
For every $A, A' \in M_{d_1}(\R)$ and $B,B' \in M_{d_2}(\R)$ we have $\|A\otimes B\|=\|A\|\cdot\|B\|$ and $(A\otimes B)(A'\otimes B')=(AA')\otimes (BB')$, see \cite[\S4.2]{HoJo94}. Let $n \geq 1$ be arbitrary and choose $i_1,\ldots,i_n \in \I$ and $j_1,\ldots,j_n \in \J$. By the preceding considerations it follows that
\begin{align*}\left\|(A_{i_1}\otimes B_{j_1})(A_{i_2}\otimes B_{j_2})\cdots (A_{i_n}\otimes B_{j_n})\right\|&=\left\|(A_{i_1}A_{i_2}\cdots A_{i_n}) \otimes (B_{j_1}B_{j_2}\cdots B_{j_n})\right\| \\
&=\left\|A_{i_1}A_{i_2}\cdots A_{i_n}\right\|\cdot \left\|B_{j_1}B_{j_2}\cdots B_{j_n}\right\|.\end{align*}
Consequently
\begin{align*}\max_{C \in \mathsf{C}_n} \|C\|&=
 \max_{\substack{i_1,\ldots,i_n \in \I\\ j_1,\ldots,j_n \in \J}}\left\|(A_{i_1}\otimes B_{j_1})(A_{i_2}\otimes B_{j_2})\cdots (A_{i_n}\otimes B_{j_n})\right\| \\
&= \left(\max_{i_1,\ldots,i_n \in \I} \left\|A_{i_1}A_{i_2}\cdots A_{i_n}\right\| \right) \left(\max_{j_1,\ldots,j_n \in \I} \left\|B_{j_1}B_{j_2}\cdots B_{j_n}\right\| \right)\\
&= \left(\max_{A \in \mathsf{A}_n} \|A\|\right)\left(\max_{B \in \mathsf{B}_n} \|B\|\right) \\
\end{align*}
for all $n \geq 1$.  It follows directly from the relevant definitions that $\varrho(\C)=\varrho(\A)\varrho(\B)=1$ and that $c_n=a_nb_n$ for every $n \geq 1$.
\end{proof}
Applying the above result inductively with $\A=\B$ and $\I=\J$ demonstrates immediately that every positive power of a marginal instability rate sequence is also a marginal instability rate sequence, but at the cost of exponentially increasing both the dimension of the matrices and the size of the indexing set $\I$. The second of these two constraints can however be easily sidestepped as follows.
\begin{proposition}\label{pr:power}
Let $d \geq1$ and let $\A\subset M_{d}(\R)$ be compact and nonempty. Suppose that $\varrho(\A)=1$. Let $k \geq 1$, define $\C:=\{A^{\otimes k}\colon A \in \mathsf{A}\}$ and denote the marginal instability rate sequences of $\A$ and $\C$ by $(a_n)$ and $(c_n)$ respectively. Then $\varrho(\C)=1$, and $c_n=a_n^k$ for every $n \geq 1$.
\end{proposition}
\begin{proof}
For every $m \geq 2$ and $A \in M_d(\R)$ we have $\|A^{\otimes m}\|=\|A\|^m$ by an elementary induction on $m$. Since $(B_1B_2)^{\otimes k} = B_1^{\otimes k}B_2^{\otimes k}$ for every $B_1,B_2 \in M_d(\R)$ it follows that for every $n \geq 1$
\[\max_{C \in \mathsf{C}_n}=\max_{A \in \mathsf{A}_n} \left\|A^{\otimes k}\right\| =\left(\max_{A \in \mathsf{A}_n} \|A\|\right)^k=a_n^k.\]
Consequently $\varrho(\mathsf{C})=\varrho(\mathsf{A})^k=1$ and $c_n=a_n^k$ for every $n \geq 1$ as required.
\end{proof}
We finally wish to show that the pairwise maximum of two marginal instability rate sequences is also a marginal instability rate sequence. We prove the following more general result:
\begin{proposition}\label{pr:maxy}
Let $\I$ be a nonempty set, let $d_1,d_2 \geq 1$ and suppose that $\A=\{A_i \colon i \in \I\}\subset M_{d_1}(\R)$ and $\B=\{B_i \colon i \in \I\}\subset M_{d_2}(\R)$ are compact and satisfy $\varrho(\A)=\varrho(\B)=1$. Let $(a_n)$ and $(b_n)$ denote the marginal instability rate sequences of $\A$ and $\B$ respectively and define $a_0, b_0:=1$. 
Let $D_i \in M_{d_1 \times d_2}(\R)$ for every $i \in \I$ and suppose that $\{\|D_i\| \colon i \in \I\}$ is bounded by a constant $K_0\geq 0$. Define
\[\C:=\left\{\begin{pmatrix} A_i&D_i \\0& B_i\end{pmatrix} \colon i\in \I\right\}\subset M_{d_1+d_2}(\R)\]
and let $(c_n)$ be the marginal instability rate sequence of $\C$. Then $\varrho(\C)=1$ and for every $n \geq 1$ we have
\[\max\{a_n, b_n\} \leq c_n \leq \max\{a_n,b_n\}+K_0\sum_{k=1}^n a_{k-1}b_{n-k}.\]
In particular, if $a_n=O(n^\alpha)$ and $b_n=O(n^\beta)$ for some $\alpha,\beta \geq 0$ then $c_n=O(n^{1+\alpha+\beta})$; and if $D_i=0$ for every $i \in \I$, then $c_n=\max\{a_n,b_n\}$ for every $n \geq 1$.
\end{proposition}
\begin{proof}
That $\varrho(\mathsf{C})=1$ follows directly from Proposition \ref{pr:uti}. A simple induction on $n$ shows that given $n \geq 1$ and $i_1,\ldots,i_n \in \I$, we have
\[C_{i_1}\cdots C_{i_n} =\begin{pmatrix} A_{i_1}\cdots A_{i_n} & \sum_{k=1}^n A_{i_1}\cdots A_{i_{k-1}}D_{i_{k}}B_{i_{k+1}} \cdots B_{i_n} \\ 0&B_{i_1}\cdots B_{i_n}\end{pmatrix}.\]
Fix $n \geq 1$ and $i_1,\ldots,i_n \in \I$. It follows using Lemma \ref{le:tricalc} that the norm $\|C_{i_1}\cdots C_{i_n}\|$ is bounded above by
\[\max\left\{\|A_{i_1}\cdots A_{i_n}\|, \|B_{i_1}\cdots B_{i_n}\|\right\} +  \sum_{k=1}^n \|A_{i_1}\cdots A_{i_{k-1}}D_{i_{k}}B_{i_{k+1}} \cdots B_{i_n}\|.\]
Now we have
\begin{align*} \sum_{k=1}^n \|A_{i_1}\cdots A_{i_{k-1}}D_{i_{k}}B_{i_{k+1}} \cdots B_{i_n}\| &\leq K_0 \sum_{k=1}^n \|A_{i_1}\cdots A_{i_{k-1}}\| \cdot\|B_{i_{k+1}} \cdots B_{i_n}\|\\
&\leq K_0\sum_{k=1}^n a_{k-1} b_{n-k}\end{align*}
using the definition of $K_0$ and the definition of the sequences $(a_n)$ and $(b_n)$, and it follows that
\[\|C_{i_1}\cdots C_{i_n}\| \leq \max\{\|A_{i_1}\cdots A_{i_n}\|, \|B_{i_1}\cdots B_{i_n}\|\}+ K_0\sum_{k=1}^n a_{k-1} b_{n-k}.\]
Since also
\[\|C_{i_1}\cdots C_{i_n}\| \geq \max\{ \|A_{i_1}\cdots A_{i_n}\|, \|B_{i_1}\cdots B_{i_n}\|\}, \]
taking the maximum over $i_1,\ldots,i_n \in \I$ yields 
\[\max\{a_n, b_n\} \leq c_n= \max_{C \in \mathsf{C}_n} \|C\| \leq \max\{a_n, b_n\}+ K_0\sum_{k=1}^n a_{k-1} b_{n-k}\]
as claimed. If $D_i=0$ for every $i \in \mathcal{I}$ then clearly we may take $K_0=0$ which yields the claimed identity $c_n=\max\{a_n,b_n\}$. If on the other hand we have $a_n=O(n^\alpha)$ and $b_n=O(n^\beta)$ for some real numbers $\alpha,\beta \geq 0$, choose $K_1, K_2>0$ such that $a_n \leq K_1(n+1)^\alpha$ and $b_n \leq K_2 (n+1)^\beta$ for all $n \geq 0$. Since for every $n \geq 1$
\begin{align*} c_n &\leq  a_n+b_n+ K_0\sum_{k=1}^n a_{k-1} b_{n-k}\\
&\leq K_1(n+1)^\alpha + K_2(n+1)^\beta + K_0K_1K_2 \sum_{k=1}^n k^\alpha (n-k+1)^\beta \end{align*}
we have
\[\limsup_{n \to \infty} n^{-1-\alpha-\beta} c_n \leq K_0K_1K_2\left( \limsup_{n \to \infty}  n^{-1-\alpha-\beta}\sum_{k=1}^n k^\alpha (n-k+1)^\beta\right),\]
and since
\begin{align*} \limsup_{n \to \infty}  n^{-1-\alpha-\beta}\sum_{k=1}^n k^\alpha (n-k+1)^\beta&=\limsup_{n \to \infty}  n^{-1-\alpha-\beta}\sum_{k=1}^n k^\alpha (n-k)^\beta\\
&=\lim_{n \to \infty}\frac{1}{n}\sum_{k=1}^n \left(\frac{k}{n}\right)^\alpha \left(1-\frac{k}{n}\right)^\beta\\
&=\int_0^1x^\alpha(1-x)^\beta dx \in (0,1) \end{align*}
by Riemann integration, it follows that $c_n=O(n^{\alpha+\beta+1})$ as required.
\end{proof}
The above result may easily be applied in combination with Corollary \ref{co:lon} to show by induction on $d$ that if $(a_n)$ is the marginal instability rate sequence of a compact nonempty set $\mathsf{A}\subset M_d(\R)$, then $O(n^{d-1})$. This result has been obtained previously on a great many occasions (see for example \cite{Be05,GuZe01,Lu06,Pr06,Su08}) and we omit this application of Proposition \ref{pr:maxy}. A minor additional step is required to deduce Theorem \ref{th:first-obe} from the above results:
\begin{proof}[Proof of Theorem \ref{th:first-obe}]
If $\mathsf{A} \subset M_{d_1}(\R)$ and $\mathsf{B} \subset M_{d_1}(\R)$ are nonempty and compact and let $(a_n)$ and $(b_n)$ be their respective marginal instability rate sequences. Define $\mathcal{I}=\mathsf{A} \times \mathsf{B}$ and for every $i=(A,B) \in \mathsf{A} \times \mathsf{B}$ write $A_i=A$, $B_i=B$. This presents $\mathsf{A}$ and $\mathsf{B}$ in a form suitable for the application of Propositions \ref{pr:stick-with-the-prod} and \ref{pr:maxy}. The application of Proposition  \ref{pr:power} is direct. Clearly if $\mathsf{A}$ and $\mathsf{B}$ are both finite, or both consist of invertible matrices, or both of these properties hold, then the same holds for the matrix sets constructed in these propositions.\end{proof}
\section{Dependence of marginal instability rate sequences on cardinality}

It is natural to ask whether the class of marginal instability rate sequences for finite sets of matrices is affected by constraints on the precise cardinality of the set of matrices. In this section we investigate the question of whether every marginal instability rate sequence for a finite set is similar to that of a set of cardinality 2. We note that a similar question in the context of the finiteness property for the joint spectral radius was investigated by R.M. Jungers and V.D. Blondel in \cite{JuBl08} and our method is inspired by their work (and by its subsequent adaptation in \cite{HaMoSi13}). While we are not able to obtain a complete answer to our question, we are able to show that any given marginal instability rate sequence  is guaranteed to be similar to that of a two-element set if it satisfies certain additional regularity conditions. These conditions are themselves investigated in a subsequent section. 

The following definition describes the regularity conditions which will be invoked below.
\begin{definition}
Let $(a_n)$ be a sequence of positive real numbers. We say that $(a_n)$ is \emph{weakly increasing} if there exists $\kappa>0$ such that $a_{n+m}\geq \kappa a_n$ for all $n,m \geq 1$ and that $(a_n)$ is \emph{weakly upper regularly varying} if for every $m \geq 1$ there exists $C_m>0$ such that $a_{nm}\leq C_m a_n$ for all $n \geq 1$.
\end{definition}
Clearly if $(a_n)$ satisfies either of these properties and $(b_n)$ is a sequence such that $a_n \simeq b_n$ then $(b_n)$ also has the corresponding property.  It is easily seen that a subadditive sequence is weakly upper regularly varying: one may take $C_m:=m$. In particular if $(a_n)$ is related to a non-decreasing sequence by the relation $\simeq$ then it is weakly increasing, and if $(a_n)$ is so related to a subadditive sequence then it is weakly upper regularly varying.

The name ``weakly upper regularly varying'' given above refers to the more standard notion of a \emph{regularly varying function} (see e.g. \cite{BiGoTe87}): a function $f \colon (0,\infty) \to (0,\infty)$ is conventionally called regularly varying if $\lim_{x\to \infty} f(\alpha x)/f(x)$ exists and is positive for every $\alpha>0$. We observe that if $f$ is a regularly varying function and $(a_n)$ satisfies $a_n=f(n)$ for every $n \geq 1$ then $(a_n)$ is weakly upper regularly varying. In particular the sequence $a_n:=n^\alpha (\log(n+1))^\beta$ is weakly upper regularly varying for every $\alpha,\beta \geq 0$. Conversely, an example of a sequence which is \emph{not} weakly upper regularly varying is the sequence $(b_n)$ defined by $b_n=n$ when $n$ is composite and $b_n=1$ when $n$ is prime, since in that case for every fixed integer $m>1$ the ratio $b_{mn}/b_n$ is unbounded as $n$ varies over the prime numbers.

The main result proved in this section is the following:
\begin{theorem}\label{th:BIG-TWOS}
Suppose that $\A=\{A_0,\ldots,A_{m-1}\} \subset M_d(\R)$ satisfies $\varrho(\A)=1$. Define two matrices $B_0, B_1 \in M_{md}(\R)$ by
\[B_0:=\begin{pmatrix}
0&0&0&\cdots &0&0&I\\
I&0&0&\cdots &0&0&0\\
0&I&0&\cdots &0&0&0&\\
\vdots&\vdots&\ddots &\ddots & &\vdots&\vdots\\

\vdots& \vdots & &\ddots &\ddots&\vdots &\vdots\\
0&0&0&\cdots&I&0&0\\
0&0&0&\cdots&0&I&0
\end{pmatrix}, \text{   }B_1:=\begin{pmatrix}
A_0&0&\cdots &0&0\\
0&A_1&\cdots &0&0\\
\vdots&\vdots &\ddots &\vdots&\vdots\\

0&0&\cdots&A_{m-2}&0\\
0&0&\cdots&0&A_{m-1}
\end{pmatrix}\]
and let $\B:=\{B_0,B_1\}\subset M_{md}(\R)$. Then $\varrho(\B)=1$, and if $(a_n)$, $(b_n)$ denote the marginal instability rate sequences of $\A$ and $\B$ respectively then the following properties hold:
\begin{enumerate}[(a)]
\item\label{it:whatever}
Let $(\gamma_n)_{n=1}^\infty$ be a weakly increasing, weakly upper regularly varying sequence of positive real numbers and $(n_\ell)_{\ell=1}^\infty$ a strictly increasing sequence of natural numbers such that the ratio $n_{\ell+1}/n_\ell$ is bounded above independently of $\ell$. Suppose that $a_n \apprle \gamma_n$ and $\gamma_{n_\ell} \apprle a_{n_\ell}$. Then $b_n \simeq \gamma_n$.
\item\label{it:you-know-what-will-get-you-you-know-where}
Suppose that $(a_n)$ is weakly increasing and weakly upper regularly varying. Then $b_n \simeq a_n$. 
\end{enumerate}
\end{theorem}
\begin{proof}
In this proof it will be convenient to extend the sequences $(a_n)$ and $(b_n)$ by defining $a_0:=1$ and $b_0:=1$. In all cases an empty product of matrices will be understood as the identity matrix, and the sets $\A_0$ and $\B_0$ will be understood to be singleton sets each containing an identity matrix of the appropriate dimension. For convenience let us say that an \emph{$(m,d)$ block permutation matrix} is a matrix in $M_{md}(\R)$ which can be written as an $m\times m$ grid of $d \times d$ matrices with at most one nonzero matrix in each of the $m$ rows and at most one nonzero matrix in each of the $m$ columns. The Euclidean operator norm of an $(m,d)$ block permutation matrix is easily seen to be the maximum of the Euclidean operator norms of its $d \times d$ blocks. Clearly the $(m,d)$ block permutation matrices form a semigroup which contains both $B_0$ and $B_1$. Given $n\geq 1$ and $i_1,\ldots,i_n \in \{0,1\}$, the product $B_{i_1}\cdots B_{i_n}$ is easily seen to be an $(m,d)$ block permutation matrix whose nonzero blocks all have the form $A_{j_1}\cdots A_{j_k}$ for some integer $k$ in the range $0 \leq k \leq n$ and for some $j_1,\ldots,j_k \in \{0,\ldots,m-1\}$, where the values of $j_1,\ldots,j_k$ depend on which block of $B_{i_1}\cdots B_{i_n}$ is being considered. It follows directly that
\[ \left\|B_{i_1}\cdots B_{i_n}\right\| \leq \max_{0 \leq k \leq n} \max_{j_1,\ldots,j_k \in \{0,\ldots,m-1\}} \left\|A_{j_1}\cdots A_{j_k}\right\|\]
and we have shown that for every $n \geq 1$
\[\max_{B \in \B_n} \|B\| \leq \max_{0 \leq k \leq n} \max_{A \in \A_k} \|A\|.\]
This result clearly also holds for $n=0$. In the other direction, given any $\ell \in \{0,\ldots,m-1\}$ we have
\[B_0^{m-\ell} B_1 B_0^\ell = \begin{pmatrix}
A_{\ell}&0&\cdots &0&0\\
0&A_{\ell+1}&\cdots &0&0\\
\vdots&\vdots &\ddots &\vdots&\vdots\\
0&0&\cdots&A_{\ell +m-2}&0\\
0&0&\cdots&0&A_{\ell+m-1}
\end{pmatrix}\]
where the indices of the matrices $A_i$ are understood modulo $m$. In particular, given any $n\geq 1$, any integer $k$ in the range $0 \leq k \leq n/(m+1)$ and any indices $i_1,\ldots,i_k \in \{0,\ldots,m-1\}$, the matrix
\[\left(B_0^{m-i_1} B_1 B_0^{i_1}\right)\left(B_0^{m-i_2} B_1 B_0^{i_2}\right)\cdots \left(B_0^{m-i_k} B_1 B_0^{i_k}\right) \in \B_{k(m+1)}\]
is an $(m,d)$ block permutation matrix whose upper left block is $A_{i_1}\cdots A_{i_k}$. This matrix therefore has norm at least $\|A_{i_1}\cdots A_{i_k}\|$. Clearly the matrix
\[\left(B_0^{m-i_1} B_1 B_0^{i_1}\right)\left(B_0^{m-i_2} B_1 B_0^{i_2}\right)\cdots \left(B_0^{m-i_k} B_1 B_0^{i_k}\right)B^{n-k(m+1)}_0 \]
has the same norm as the former matrix and belongs to $\B_n$. This suffices for us to deduce 
\[ \max_{B \in \B_n}\|B\| \geq \max_{0 \leq k \leq \lfloor n/(m+1)\rfloor} \max_{A \in \A_k} \|A\|\]
and we have proved the bounds
\begin{equation}\label{eq:when-a-man-say-ow}\max_{0 \leq k \leq \lfloor n/(m+1)\rfloor} \max_{A \in \A_k} \|A\| \leq \max_{B \in \B_n}\|B\| \leq \max_{0 \leq k \leq n} \max_{A \in \A_k} \|A\|\end{equation}
for every $n \geq 0$.

To see that $\varrho(\B)=1$ we argue as follows. Given any $\varepsilon>0$, since $\varrho(\A)=1$ we may choose $L_\varepsilon>0$ such that $\max_{A \in \A_n}\|A\|\leq L_\varepsilon e^{n\varepsilon}$ for every $n\geq 1$. Consequently
\begin{align*}\varrho(\B)=\limsup_{n \to \infty} \max_{B \in \B_n} \|B\|^{\frac{1}{n}}
& \leq \limsup_{n \to \infty} \max_{0 \leq k \leq n} \max_{A \in \A_k} \|A\|^{\frac{1}{n}}\\
& \leq  \limsup_{n \to \infty} \max_{0 \leq k \leq n}\left(L_\varepsilon e^{k\varepsilon}\right)^{\frac{1}{n}}  =  \lim_{n \to \infty} \left(L_\varepsilon e^{n\varepsilon}\right)^{\frac{1}{n}} = e^\varepsilon\end{align*}
and since $\varepsilon>0$ was arbitrary we have $\varrho(\B)\leq 1$. On the other hand clearly $\max_{B \in \B_n} \|B\|^{1/n} \geq \|B_0^n\|^{1/n}=1$ for every $n \geq 1$ and it follows that $\varrho(\B)=1$ as claimed. 

We now prove \eqref{it:whatever}. The hypotheses imply that we may choose $\kappa>0$ such that $\gamma_{n_2} \geq \kappa \gamma_{n_1}$ whenever $1 \leq n_1 \leq n_2$, an integer $p$ such that $(m+1)n_{\ell+1} \leq pn_\ell$ for all $\ell \geq 1$, and a real constant $K>0$ such that $\gamma_{n_\ell} \leq K a_{n_\ell}$ for all $\ell \geq 1$ and such that $a_n \leq K\gamma_n$ and $\gamma_{pn} \leq K\gamma_n$ for all $n \geq 1$. Let $k \geq 1$ and suppose that $(m+1)n_\ell \leq k < (m+1)n_{\ell+1}$ for some $\ell \geq 1$. Using the weakly increasing and weakly upper regularly varying properties of $(\gamma_n)$ we have
\[\gamma_{(m+1)n_{\ell+1}} \leq \kappa^{-1}\gamma_{pn_\ell} \leq K\kappa^{-1}\gamma_{n_\ell}.\]
Consequently, using the lower bound from \eqref{eq:when-a-man-say-ow} and the weakly increasing property once more,
\[b_k \geq  \max_{0 \leq j < n_{\ell+1}} a_j\geq a_{n_\ell} \geq K^{-1} \gamma_{n_\ell}\geq K^{-2} \kappa \gamma_{(m+1)n_{\ell+1}} \geq  K^{-2}\kappa^2  \gamma_{k}.\]
The upper bound from \eqref{eq:when-a-man-say-ow} yields 
\begin{align*}b_k &\leq \max_{0 \leq k < (m+1)n_{\ell+1}} a_k\\
 &\leq K\max_{0 \leq k < (m+1)n_{\ell+1}} \gamma_k \leq  K\kappa^{-1} \gamma_{pn_\ell} \leq K^2\kappa^{-1}\gamma_{n_\ell} \leq K^2\kappa^{-2}\gamma_k\end{align*}
and we conclude that for every $\ell\geq 1$
\[K^{-2}\kappa^2 \leq \min_{(m+1)n_\ell \leq k <(m+1)n_{\ell+1}} \frac{b_k}{\gamma_k} \leq\max_{(m+1)n_\ell \leq k <(m+1)n_{\ell+1}}\frac{b_k}{\gamma_k} \leq K^2\kappa^{-2}.\]
Hence
\[K^{-2}\kappa^2 \leq \inf_{k \geq (m+1)m_1} \frac{b_k}{\gamma_k} \leq \sup_{k \geq (m+1)n_1} \frac{b_k}{\gamma_k} \leq K^2\kappa^{-2}\]
from which it follows that directly that
\[0< \inf_{k \geq 1} \frac{b_k}{\gamma_k} \leq \sup_{k \geq 1} \frac{b_k}{\gamma_k}<\infty\]
as required. To prove \eqref{it:you-know-what-will-get-you-you-know-where} we simply apply \eqref{it:whatever} with $n_\ell := \ell$ and $\gamma_n:=a_n$. 
\end{proof}

%
%

\section{Regularity properties of marginal instability rate sequences}

If we knew \emph{a priori} that every marginal instability rate sequence is both weakly increasing and weakly upper regularly varying then Theorem \ref{th:BIG-TWOS} would imply that every marginal instability rate sequence can be realised by a set of matrices with cardinality 2. In this section we investigate those two properties for general marginal instability rate sequences. Surprisingly we are unable to determine conclusively whether or not every marginal instability rate sequence has either property, but we give a sufficient condition for the second.
A version of the following result was used implicitly in \cite{Mo22}.
\begin{theorem}
Let $\A=\{A_i \colon i \in \I\} \subset M_d(\R)$ be a compact and nonempty set such that $\varrho(\A)=1$. Suppose that there exist an integer $k$ in the range $1 \leq k <d$, an invertible matrix $X \in M_d(\R)$ and sets of matrices $\B=\{B_i \colon i \in \I\}\subseteq M_k(\R)$,  $\C=\{C_i \colon i \in \I\}\subset M_{d-k}(\R)$ and $\D=\{D_i \colon i \in \I\}\subset M_{k \times (d-k)}(\R)$ such that we may write
\[A_i=X^{-1}\begin{pmatrix} B_i & D_i \\ 0&C_i\end{pmatrix}X \]
for every $i \in \I$. Suppose that both $\B$ and $\C$ are product bounded, and let $(a_n)$ denote the marginal instability rate sequence of $\A$. Then there exists a positive subadditive sequence $(\alpha_n)_{n=1}^\infty$ such that $a_n \simeq 1+\alpha_n$. In particular $(a_n)$  is weakly upper regularly varying, and the sequence $(a_n/n)$ either converges to zero or is bounded away from zero and infinity. 
\end{theorem}
\begin{proof}
Since both $\B$ and $\C$ are product bounded, by Proposition \ref{pr:irr}(ii) we may choose norms $\threebar{\cdot}_1$ and $\threebar{\cdot}_2$ on $\R^k$ and $\R^{d-k}$ such that $\max_{i \in \I} \threebar{B_i}_1\leq 1$ and $\max_{i \in \I}\threebar{C_i}_2 \leq 1$ respectively. Furthermore, for every $n \geq 1$ we have
\[1 =\varrho(\mathsf{A})=\max\{\varrho(\mathsf{B}), \varrho(\mathsf{C})\} \leq \max_{i_1,\ldots, i_n \in \I} \max\left\{ \threebar{B_{i_1}\cdots B_{i_n}}_1, \threebar{C_{i_1}\cdots C_{i_n}}_2\right\}\]
from Proposition \ref{pr:uti}, subadditivity and the definition of the joint spectral radius. Consequently
\begin{equation}\label{eq:banan}\max_{i_1,\ldots, i_n \in \I} \max\left\{ \threebar{B_{i_1}\cdots B_{i_n}}_1, \threebar{C_{i_1}\cdots C_{i_n}}_2\right\}=1\end{equation}
for every $n \geq 1$.

Let $\threebar{\cdot}$ denote the norm on $\R^d\simeq \R^k \oplus \R^{d-k}$ given by $\threebar{u\oplus v}:=\threebar{u}_1+\threebar{v}_2$, and for every $Z \in M_{k \times (d-k)}(\R)$ define
\[\threebar{Z}_*:=\max\left\{\threebar{Zv}_1 \colon v \in \R^{d-k}\text{ and }\threebar{v}_2=1\right\}.\]
Let $(a_n)$ be the marginal instability rate sequence of $\A$ computed with respect to the Euclidean norm, and for every $n \geq 1$ define
\[\hat{a}_n:= \max_{i_1,\ldots,i_n \in \I} \threebar{XA_{i_1}\cdots A_{i_n}X^{-1}},\]
\[\alpha_n:=  \max_{i_1,\ldots,i_n \in \I}\threebar{\sum_{j=1}^n B_{i_1}\cdots B_{i_{j-1}} D_{i_j}C_{i_{j+1}}\cdots C_{i_n}}_*.\]
Clearly $(\hat{a}_n)\simeq (a_n)$. Given $n \geq 1$ and $i_1,\ldots,i_n \in \I$, we have
\begin{align*}
{\lefteqn{\threebar{XA_{i_1}\cdots A_{i_n}X^{-1}}}} &\\
& =\threebar{\begin{pmatrix} B_{i_1}\cdots B_{i_n} & \sum_{j=1}^n B_{i_1}\cdots B_{i_{j-1}} D_{i_j}C_{i_{j+1}}\cdots C_{i_n} \\ 0&C_{i_1}\cdots C_{i_n}\end{pmatrix}}\\
&=\max\left\{ \threebar{B_{i_1}\cdots B_{i_n}}_1, \threebar{\sum_{j=1}^n B_{i_1}\cdots B_{i_{j-1}} D_{i_j}C_{i_{j+1}}\cdots C_{i_n}}_* +\threebar{C_{i_1}\cdots C_{i_n}}_2\right\}\end{align*}
and in view of \eqref{eq:banan} this implies $\max\{1, \alpha_n\} \leq \hat{a}_n \leq 1+\alpha_n$ for every $n \geq 1$. Thus $a_n \simeq \hat{a}_n \simeq 1+\alpha_n$ as needed.
To complete the proof it suffices to show that $\alpha_{n+m} \leq \alpha_n  +\alpha_m$ for every $n,m \geq 1$. Given $n,m \geq 1$ and $i_1,\ldots,i_{n+m} \in \I$ we may write
\begin{align*}{\lefteqn{\sum_{j=1}^{n+m} B_{i_1}\cdots B_{i_{j-1}} D_{i_j}C_{i_{j+1}}\cdots C_{i_{n+m}}}}&\\
&=
\begin{aligned}[t]
&\sum_{j=1}^{m} B_{i_{m+1}}\cdots B_{i_{j-1}} D_{i_j}C_{i_{j+1}}\cdots C_{i_{n+m}}\\
&+ B_{i_1}\cdots B_{i_m} \sum_{j=m+1}^{n+m} B_{i_{m+1}}\cdots B_{i_{j-1}} D_{i_j}C_{i_{j+1}}\cdots C_{i_{n+m}}
\end{aligned}
\end{align*}
and therefore
\begin{align*} {\lefteqn{\threebar{\sum_{j=1}^{n+m} B_{i_1}\cdots B_{i_{j-1}} D_{i_j}C_{i_{j+1}}\cdots C_{i_{n+m}}}_*}}&\\ 
&\leq 
\begin{aligned}[t]
& \threebar{\sum_{j=1}^{m} B_{i_{m+1}}\cdots B_{i_{j-1}} D_{i_j}C_{i_{j+1}}\cdots C_{i_{n+m}}}_*\\
&+ \threebar{B_{i_1}\cdots B_{i_m}}_1\cdot \threebar{\sum_{j=m+1}^{n+m} B_{i_{m+1}}\cdots B_{i_{j-1}} D_{i_j}C_{i_{j+1}}\cdots C_{i_{n+m}}}_*
\end{aligned}\\
&\leq 
\begin{aligned}[t]
& \threebar{\sum_{j=1}^{m} B_{i_{m+1}}\cdots B_{i_{j-1}} D_{i_j}C_{i_{j+1}}\cdots C_{i_{n+m}}}_*\\
& + \threebar{\sum_{j=m+1}^{n+m} B_{i_{m+1}}\cdots B_{i_{j-1}} D_{i_j}C_{i_{j+1}}\cdots C_{i_{n+m}}}_* \end{aligned}\\
&\leq \alpha_m + \alpha_n.\end{align*}
By taking the maximum over $i_1,\ldots,i_{n+m} \in \I$ we obtain $\alpha_{n+m}\leq \alpha_n + \alpha_m$ as required. The proof is complete.\end{proof}
For the property of being weakly increasing we have been able to make much less progress. The following much weaker statement is nonetheless surprisingly subtle to prove. 
\begin{theorem}\label{th:no-sudden-moves}
Let $\A \subset M_d(\R)$ be compact and nonempty with $\varrho(\A)=1$ and let $(a_n)$ be its marginal instability rate sequence. Then there exists $\kappa>0$ such that $a_{n+1} \geq \kappa a_n$ for every $n \geq 1$.
\end{theorem}
\begin{proof}[Proof of Theorem \ref{th:no-sudden-moves}]
We will prove the theorem by induction on $d$. In the case $d=1$ we have $\max_{A \in \A_n}\|A_n\|=1$ for every $n$ and the result is trivial. For the remainder of the proof we fix an integer $d>1$ and a compact nonempty set $\A =\{A_i \colon i \in \I\} \subset M_d(\R)$ and suppose that all cases of the theorem in dimension strictly smaller than $d$ have been proved. If $\A$ is irreducible then  $\max_{A \in \A_n}\|A\| \simeq 1$ by Corollary \ref{co:lon} and the result follows, so we further  suppose that $\A$ is reducible. Choose an invertible matrix $X \in M_d(\R)$, integer $k \geq 1$ and sets of matrices $\B=\{B_i \colon i \in \I\}\subset M_k(\R)$, $\D=\{D_i \colon i \in \I\}\subset M_{k \times (d-k)}(\R)$ and $\C=\{C_i\colon i \in \I\}\subset M_{d-k}(\R)$ such that for every $i \in \I$ we have
\begin{equation}\label{eq:zeds}A_i=X^{-1}\begin{pmatrix} B_i & D_i \\ 0 &C_i\end{pmatrix}X.\end{equation}
We moreover take $k$ to be the smallest possible integer such that this construction is possible. It is clear that $\mathsf{B}$, $\mathsf{C}$ and $\mathsf{D}$ are compact since they are continuous images of $\mathsf{A}$.
 We observe that $\B$ is necessarily irreducible: if this were not the case then by a further change of basis we could write
\[A_i=X^{-1}Y^{-1}\begin{pmatrix} B_i^{(1)} & B^{(2)}_i & D_i^{(1)} \\
0&B_i^{(2)}& D_i^{(2)} \\
\\ 0 &0&C_i\end{pmatrix}YX =X^{-1}Y^{-1} \begin{pmatrix} B_i^{(1)} & D_i'\\ 0&C_i'\end{pmatrix}YX,\] 
for every $i \in \I$ where $B_i^{(1)} \in M_\ell(\R)$, $B_i^{(2)} \in M_{\ell \times (k-\ell)}(\R)$, $D_i^{(1)} \in M_{\ell \times( d-k)}(\R)$, $D_i^{(2)} \in M_{(k-\ell) \times (d-k)}(\R)$, $D_i' \in M_{\ell \times (d-\ell)}(\R)$ and $C_i' \in M_{d-\ell}(\R)$, and where $\ell$ is an integer in the range $1 \leq \ell<k$. This would contradict the minimality of $k$ and is therefore impossible; we conclude that $\B$ must be irreducible as claimed.
 We will find it convenient to equip $M_d(\R)$ with the norm
\[\threebar{X^{-1}\begin{pmatrix} Z_1&Z_2 \\ Z_3&Z_4\end{pmatrix}X }:=\|Z_1\|+\|Z_2\|+\|Z_3\|+\|Z_4\|\]
where the dimensions of the block matrices $Z_i$ match those in \eqref{eq:zeds}. For every $n \geq 1$ and $i_1,\ldots,i_n \in \I$ this yields
\begin{align}\label{eq:yt}\threebar{A_{i_1}\cdots A_{i_n}} &= \threebar{X^{-1}\begin{pmatrix}B_{i_1}\cdots B_{i_n} & \sum_{j=1}^n B_{i_1}\cdots B_{i_{j-1}} D_{i_j} C_{i_{j+1}} \cdots C_{i_n} \\ 0& C_{i_1}\cdots C_{i_n}\end{pmatrix}X} \\\nonumber
&=\|B_{i_1}\cdots B_{i_n}\| +\left\|\sum_{j=1}^n B_{i_1}\cdots B_{i_{j-1}} D_{i_j} C_{i_{j+1}} \cdots C_{i_n}\right\|+\| C_{i_1}\cdots C_{i_n}\|. \end{align}
For the remainder of the proof we define for every $n \geq 1$
\[a_n:=\max_{i_1,\ldots,i_n \in \I} \threebar{A_{i_1}\cdots A_{i_n}}.\]

We proceed through a series of cases. By Proposition \ref{pr:uti} we have $\varrho(\A)=\max\{\varrho(\B), \varrho(\C)\}$ and in particular $\varrho(\C) \leq 1$. The first case which we consider is that in which $\varrho(\C)<1$. 
Since $\B$ is irreducible, by Proposition \ref{pr:irr}(i) there exists $K_1>0$ such that $\max_{B \in \B_n} \|B\| \leq K_1$ for every $n \geq 1$. Clearly there exists $K_2>0$ such that $\|D_i\|\leq K_2$ for every $i \in \I$, and since $\varrho(\C)<1$ we may choose $K_3>0$ and $\theta \in (0,1)$ such that $\max_{C \in \C_n} \|C\| \leq K_3\theta^n$ for every $n \geq 1$. It follows that for every $n \geq 1$ and $i_1,\ldots,i_n \in \I$ 
\begin{align*}\threebar{A_{i_1}\cdots A_{i_n}} &= \|B_{i_1}\cdots B_{i_n}\| +\left\|\sum_{j=1}^n B_{i_1}\cdots B_{i_{j-1}} D_{i_j} C_{i_{j+1}} \cdots C_{i_n}\right\|+\| C_{i_1}\cdots C_{i_n}\| \\
&\leq K_1 + \sum_{j=1}^n K_1K_2K_3 \theta^{n-j} + K_3\theta^n\\
& < K_1+ \frac{K_1K_2K_3}{1-\theta} +K_3.\end{align*}
We conclude that $(a_n)$ is bounded. Since also $\inf_{n\geq 1} a_n>0$ by Corollary \ref{co:lon} we deduce that $(a_n) \simeq 1$, and the conclusion of the theorem clearly applies to $(a_n)$ in this case.

In the remaining cases we have  $\varrho(\C)=1$. By the induction hypothesis there exists $\delta \in (0,1]$ such that the sequence
\[c_n:=\max_{i_1,\ldots,i_n \in \I} \|C_{i_1}\cdots C_{i_n}\|\]
satisfies $c_{n+1} \geq \delta c_n$ for every $n \geq 1$. Since $\mathsf{B}$ is irreducible, by Lemma \ref{le:irr-e} there exists $\varepsilon \in (0,1]$ such that $\max_{i \in \I} \|B_i M\| \geq \varepsilon \|M\|$ whenever $M$ is a real matrix with $k$ rows. Define also $K:=\max\{1, \max_{i \in \I} \|D_i\|\}$. We fix these constants $\varepsilon, \delta, K$ for the remainder of the proof.

 Fix $n \geq 1$ and choose $i_1,\ldots,i_n \in \I$ such that $a_n=\threebar{A_{i_1}\cdots A_{i_n}}$. If we have
\[\|B_{i_1}\cdots B_{i_n}\| \geq \frac{1}{2}\threebar{A_{i_1}\cdots A_{i_n}}\]
then using Lemma \ref{le:irr-e} we may choose $i_0 \in \I$ such that
 \[\|B_{i_0}\cdots B_{i_n}\| \geq \varepsilon \|B_{i_1}\cdots B_{i_n}\|\]
and therefore
\begin{equation}\label{eq:ea}a_{n+1} \geq \threebar{A_{i_0}\cdots A_{i_n}} \geq \|B_{i_0}\cdots B_{i_n}\| \geq \varepsilon \|B_{i_1}\cdots B_{i_n}\| \geq \frac{\varepsilon}{2}\threebar{A_{i_1}\cdots A_{i_n}}=\frac{\varepsilon}{2}a_n.\end{equation}
 If instead we have both
\begin{equation}\label{eq:hiro}\|B_{i_1}\cdots B_{i_n}\| < \frac{1}{2}\threebar{A_{i_1}\cdots A_{i_n}}\end{equation}
and
\[\left\|\sum_{j=1}^n B_{i_1}\cdots B_{i_{j-1}} D_{i_j} C_{i_{j+1}} \cdots C_{i_n}\right\|>\left(\frac{2K}{\varepsilon}+1\right)\|C_{i_1}\cdots C_{i_n}\|\]
then we again apply Lemma \ref{le:irr-e}, this time to choose $i_0 \in \I$ such that 
\[\left\|B_{i_0}\sum_{j=1}^n B_{i_1}\cdots B_{i_{j-1}} D_{i_j} C_{i_{j+1}} \cdots C_{i_n}\right\|\geq \varepsilon \left\|\sum_{j=1}^n B_{i_1}\cdots B_{i_{j-1}} D_{i_j} C_{i_{j+1}} \cdots C_{i_n}\right\|.\]
We now have
\begin{align}\label{eq:eb} a_{n+1} &\geq \threebar{A_{i_0}\cdots A_{i_n}}\\\nonumber
&= \threebar{X^{-1}\begin{pmatrix}B_{i_0}&D_{i_0} \\ 0&C_{i_0}\end{pmatrix}\begin{pmatrix}B_{i_1}\cdots B_{i_n} & \sum_{j=1}^n B_{i_1}\cdots B_{i_{j-1}} D_{i_j} C_{i_{j+1}} \cdots C_{i_n} \\ 0& C_{i_1}\cdots C_{i_n}\end{pmatrix}X}\\\nonumber
&\geq \left\|B_{i_0}\left(\sum_{j=1}^n B_{i_1}\cdots B_{i_{j-1}} D_{i_j} C_{i_{j+1}} \cdots C_{i_n}\right) + D_{i_0}C_{i_1}\cdots C_{i_n}\right\|\\\nonumber
&\geq \varepsilon \left\|\sum_{j=1}^n B_{i_1}\cdots B_{i_{j-1}} D_{i_j} C_{i_{j+1}} \cdots C_{i_n}\right\| - K\|C_{i_1}\cdots C_{i_n}\| \\\nonumber
&>\frac{\varepsilon}{2}\left(\left\|\sum_{j=1}^n B_{i_1}\cdots B_{i_{j-1}} D_{i_j} C_{i_{j+1}} \cdots C_{i_n}\right\| + \|C_{i_1}\cdots C_{i_n}\| \right) \\\nonumber
&>\frac{\varepsilon}{4} \threebar{A_{i_1}\cdots A_{i_n}}=\frac{\varepsilon}{4} a_n\end{align}
where we have used the inequality
\begin{equation}\label{eq:prota}\threebar{A_{i_1}\cdots A_{i_n}} \leq 2\left(\left\|\sum_{j=1}^n B_{i_1}\cdots B_{i_{j-1}} D_{i_j} C_{i_{j+1}} \cdots C_{i_n}\right\| + \|C_{i_1}\cdots C_{i_n}\| \right)\end{equation}
which follows from the combination of \eqref{eq:yt} with \eqref{eq:hiro}. Finally, in the case in which
\[\|B_{i_1}\cdots B_{i_n}\| < \frac{1}{2}\threebar{A_{i_1}\cdots A_{i_n}}\]
and
\[\left\|\sum_{j=1}^n B_{i_1}\cdots B_{i_{j-1}} D_{i_j} C_{i_{j+1}} \cdots C_{i_n}\right\|\leq\left(\frac{2K}{\varepsilon}+1\right)\|C_{i_1}\cdots C_{i_n}\|,\]
we find that
\begin{align}\label{eq:ec}a_n &=\threebar{A_{i_1}\cdots A_{i_n}}\\\nonumber
& \leq 2\left(\left\|\sum_{j=1}^n B_{i_1}\cdots B_{i_{j-1}} D_{i_j} C_{i_{j+1}} \cdots C_{i_n}\right\|+\|C_{i_1}\cdots C_{i_n}\|\right)\\\nonumber
&\leq \left(\frac{4K}{\varepsilon}+4\right) \|C_{i_1}\cdots C_{i_n}\|\\\nonumber
&\leq  \left(\frac{4K}{\varepsilon}+4\right)c_n\\\nonumber
&\leq  \left(\frac{4K}{\varepsilon}+4\right)\delta^{-1} c_{n+1}\\\nonumber
&\leq  \left(\frac{4K}{\varepsilon}+4\right)\delta^{-1}a_{n+1}\end{align}
where we have used \eqref{eq:prota}, the induction hypothesis and the elementary inequality $a_m \geq c_m$ which holds for every $m \geq 1$ by virtue of \eqref{eq:yt}. Combining the three estimates \eqref{eq:ea}, \eqref{eq:eb} and \eqref{eq:ec} it follows that in the case where $\varrho(\C)=1$ we always have
\[a_{n+1} \geq \min\left\{\frac{\varepsilon}{2}, \frac{\varepsilon}{4}, \frac{\delta\varepsilon}{4K+4\varepsilon}\right\} a_n\geq\left(\frac{\delta \varepsilon}{4K+4}\right)a_n\]
for every $n \geq 1$. This completes the proof of the induction step and the theorem is proved.
\end{proof}

%
%

\section{Two extensions of an example of Protasov and Jungers}

In \cite[Example 3.1]{GuZe01} N. Guglielmi and M. Zennaro showed that for every $\alpha \in (0,1)$ the marginal instability rate sequence $(a_n)$ of the compact infinite set $\A \subset M_2(\R)$ defined by
\[\A=\left\{ \begin{pmatrix}
1&\theta \\ 0 &1-\theta^{\frac{1}{1-\alpha}}
\end{pmatrix} \colon \theta \in [0,1]
\right\}\]
satisfies $a_n \simeq n^\alpha$. In particular this demonstrated that marginal instability rate sequences $(a_n)$ may grow as non-integer powers of $n$. The first illustration of a similar phenomenon for \emph{finite} sets $\A$ was given by Protasov and Jungers in \cite{PrJu15}: they constructed a pair of matrices $\A=\{A_0,A_1\} \subset M_3(\R)$ for which the marginal instability rate sequence $(a_n)$ satisfies $a_n \apprle n^{1/3}$ and also satisfies $n_\ell^{1/3} \apprle a_{n_\ell}$ along a certain sparse subsequence $(n_\ell)$. In this section we will extend Protasov and Jungers' example by allowing different values for the exponent, and also show how this example can be extended to give the stronger result $a_n \simeq n^{1/3}$ at the cost of increasing the dimension from $3$ to $6$. We prove:
\begin{theorem}\label{th:pyjama}
Let $\alpha \in [\frac{1}{3},\frac{1}{2})$ and let $\|\cdot\|$ be the Euclidean norm on $\R^3$. Then there exists $\theta \in \R$ such that for the pair of matrices $\A=\{A_0, A_1\}$ defined by
\[A_0:=\begin{pmatrix}1&0&0\\ 0&1&0\\ 0&0&0\end{pmatrix},\qquad  A_1:=\begin{pmatrix}1&\sin \theta&\cos\theta -1\\ 0&\cos \theta & -\sin\theta \\ 0&\sin\theta &\cos\theta\end{pmatrix}\]
the marginal instability rate sequence $(a_n)$ of $\A$ satisfies $a_n \apprle n^\alpha$, and furthermore along a certain subsequence $(n_\ell)_{\ell=1}^\infty$ we have $n_{\ell}^\alpha \apprle a_{n_\ell}$. In the case $\alpha=\frac{1}{3}$ we may choose the subsequence $(n_\ell)$ in such a way that the sequence $(n_{\ell+1}/n_\ell)$ is bounded. 
\end{theorem}
Combining the above with Theorem \ref{th:BIG-TWOS} yields:
\begin{corollary}\label{co:harvey}
There exists $\theta \in \R$ such that the set $\B=\{B_0,B_1\}\subset M_6(\R)$ defined by
\[B_0:=\begin{pmatrix}
0&0&0&1&0&0\\
0&0&0&0&1&0\\
0&0&0&0&0&1\\
1&0&0&0&0&0\\
0&1&0&0&0&0\\
0&0&1&0&0&0\end{pmatrix},
\qquad
B_1:=\begin{pmatrix}
1&0&0&0&0&0\\
0&1&0&0&0&0\\
0&0&0&0&0&0\\
0&0&0&1&\sin \theta&\cos\theta -1\\ 
0&0&0&0&\cos \theta & -\sin\theta \\
0&0&0&0&\sin\theta &\cos\theta\end{pmatrix}
\]
satisfies $\varrho(\B)=1$ and has marginal instability rate sequence $(b_n)$ satisfying $b_n \simeq n^{\frac{1}{3}}$.
\end{corollary}
\begin{proof}
Let $A_0,A_1$ and $(n_\ell)_{\ell=1}^\infty$ be as given by Theorem \ref{th:pyjama} with $\alpha=1/3$, where we note that the sequence $(n_{\ell+1}/n_\ell)_{\ell=1}^\infty$ is bounded. Apply Theorem \ref{th:BIG-TWOS}(a) with $(n_\ell)$ as given by Theorem \ref{th:pyjama} and with $\gamma_n:=n^{\frac{1}{3}}$.
\end{proof}
Our strategy of proof for Theorem \ref{th:pyjama} will follow that of \cite[Theorem 3]{PrJu15} and requires two lemmas. Our first lemma generalises an estimate used by Protasov and Jungers \cite[Lemma 3]{PrJu15}:
\begin{lemma}\label{le:bend}
Let $\beta>0$ and $0<\delta\leq 1$. Then there exists $C>0$ such that for all $\phi \in \R \setminus \pi\Z$ and $p \in [\delta,\infty)$
\[|\sin \phi| + p|\cos \phi| \leq \left(p^{2+\beta} + \frac{C}{|\sin \phi|^\beta} \right)^\frac{1}{2+\beta}.\]
\end{lemma}
\begin{proof}
We claim that the quantity
\[K:=\sup_{p \in [\delta,\infty)} p^{-\beta} \left(\left(p+\frac{2}{p}\right)^{2+\beta}-p^{2+\beta}\right)\]
is finite. Clearly it suffices to show that the function 
\[p\mapsto  p^{-\beta} \left(\left(p+\frac{2}{p}\right)^{2+\beta}-p^{2+\beta}\right)\]
extends continuously to the compact space $[\delta,\infty]$, so the claim follows if the limit 
\[\lim_{p \to \infty} p^{-\beta} \left(\left(p+\frac{2}{p}\right)^{2+\beta}-p^{2+\beta}\right)\]
exists and is finite. But we have
\begin{align*}{\lefteqn{\lim_{p \to \infty} p^{-\beta} \left(\left(p+\frac{2}{p}\right)^{2+\beta}-p^{2+\beta}\right)}}&\\
 &=\lim_{p \to \infty} \left(1+\frac{2}{p^2}\right)^\beta\left(p+\frac{2}{p}\right)^2-p^2 \\
&=\lim_{p \to \infty} p^2 \left(\left(1+\frac{2}{p^2}\right)^\beta-1\right) +4\left(1+\frac{2}{p^2}\right)^\beta + \frac{4}{p^2} \left(1+\frac{2}{p^2}\right)^\beta\\
&=2\beta+4\end{align*}
where the limit
\[\lim_{p \to \infty} p^2 \left(\left(1+\frac{2}{p^2}\right)^\beta-1\right)=\lim_{\varepsilon \searrow 0} \frac{(1+2\varepsilon)^\beta - 1}{\varepsilon}=2\beta\]
follows from elementary calculus, and we conclude that the quantity $K$ is finite as claimed.

We may now prove the lemma. Define $C:=2^\beta K>0$. Let $\phi \in \R\setminus \pi\Z$ and $p \geq \delta$ and define $s:=|\sin \phi|$. If $p \geq \delta$ satisfies $ps>2$ then we have
\[|\sin \phi| + p|\cos \phi| = s+p\sqrt{1-s^2} \leq s+p - \frac{ps^2}{2} <p<\left(p^{2+\beta} + \frac{C}{|\sin \phi|^\beta} \right)^\frac{1}{2+\beta}\]
as required. Otherwise $s \leq 2/p$ and therefore
\[s^\beta\left((p+s)^{2+\beta} -  p^{2+\beta}\right) \leq  2^\beta p^{-\beta}\left(\left(p+\frac{2}{p}\right)^{2+\beta}-p^{2+\beta}\right) \leq 2^\beta K =C.\]
By elementary rearrangements this yields
\[s+p \leq \left(p^{2+\beta}+Cs^{-\beta}\right)^{\frac{1}{2+\beta}}\]
and it follows that
\[|\sin \phi| + p|\cos \phi| < s+p \leq \left(p^{2+\beta} + \frac{C}{|\sin \phi|^\beta} \right)^\frac{1}{2+\beta}\]
as required. The proof is complete.
\end{proof}
We also require the following elementary number-theoretic construction:
\begin{lemma}\label{le:dio}
For every $\gamma \geq1$ there exist $\theta \in \R$ and a strictly increasing sequence of natural numbers $(n_\ell)_{\ell=1}^\infty$ such that
\[\inf_{n \geq 1} n^\gamma |\sin n\theta|>0, \qquad  \sup_{\ell \geq 1}  n_\ell^{\gamma}|\sin n_\ell\theta|<\infty.\]
In the case $\gamma=1$ we may additionally choose $(n_\ell)$ such that the sequence $(n_{\ell+1}/n_\ell)$ is bounded.
\end{lemma}
\begin{proof}
For any $x \in \R$ and any nonempty closed set $K \subset \R$ let $\dist(x,K)$ denote the distance from $x$ to the nearest element of $K$. Since $(2/\pi)\dist (x,\pi \Z) \leq |\sin x| \leq \dist(x,\pi\Z)$ for all $x \in \R$, to prove the lemma it suffices to find $\vartheta \in \R$ and a strictly increasing sequence of natural numbers $(q_n)_{n=1}^\infty$ such that  
\begin{equation}\label{eq:cf1}\inf_{n \geq 1} n^{\gamma}\dist(n\vartheta, \Z)>0\end{equation} and
\begin{equation}\label{eq:cf2}\sup_{n \geq 1} q_n^{\gamma}\dist(q_n\vartheta, \Z)<\infty\end{equation}
since we may then define $\theta:=\pi\vartheta$ and $n_\ell:=q_\ell$ for every $\ell \geq 1$. We will construct $\vartheta$ via its continued fraction expansion. We will refer to Khinchin's classic text \cite{Kh97} for the relevant general results concerning continued fractions.

Define $q_{-1}:=0$, $p_{-1}:=1$, $q_0:=1$ and $p_0:=0$ and define sequences of natural numbers $(a_n)_{n=1}^\infty$, $(p_n)_{n=-1}^\infty$ and $(q_n)_{n=-1}^\infty$ inductively as follows. For each $n \geq 1$ define $a_n:=\lceil q_{n-1}^{\gamma-1} \rceil \in \N$, $p_n:=a_n p_{n-1}+p_{n-2}$ and $q_n:= a_n q_{n-1}+q_{n-2}$.  By Khinchin's Theorems 2 and 10 the resulting sequence $(p_n/q_n)$ converges to a limit $\vartheta \in (0,1)$. Moreover, for every $n \geq 1$ we have by Khinchin's Theorems 9 and 13
\[\frac{1}{2q_{n+1}}<|q_n \vartheta - p_n| <\frac{1}{q_{n+1}},\]
and if $k \in \N$ is not equal to any of the integers $q_n$ then by Khinchin's Theorem 19
\begin{equation}\label{eq:k19} \dist(k\vartheta,\Z)>\frac{1}{2k}.\end{equation}
The sequence $(q_n)_{n=1}^\infty$ satisfies $q_1=1$ and is strictly increasing as a simple consequence of its definition, so for every $k \in \N$ there necessarily exists a unique $n\in \N$ such that $q_n \leq k<q_{n+1}$. If $n \geq 1$ then the inequality $|q_n\vartheta-p_n|<1/q_{n+1}\leq 1/2$ implies that $p_n$ is the closest integer to $q_n\vartheta$, so in this case $\dist(q_n\vartheta,\Z)=|q_n\vartheta-p_n|>1/2q_{n+1}$. Combining this with \eqref{eq:k19} we conclude that for every $n \in \N$, for every $k$ in the range $q_n \leq k <q_{n+1}$ we have
\[\dist(k\vartheta,\Z)>\frac{1}{2q_{n+1}}.\]
Now let $k \in \N$ be arbitrary and choose $n\in \N$ such that $q_n\leq k<q_{n+1}$. We have
\[\dist(k\vartheta,\Z)>\frac{1}{2q_{n+1}} =\frac{1}{2(\lceil q_n^{\gamma-1} \rceil q_n+q_{n-1})}>\frac{1}{2(q_n^{\gamma}+q_n+q_{n-1})}\geq\frac{1}{6q_n^{\gamma}}\geq \frac{1}{6k^{\gamma}}\]
and this proves \eqref{eq:cf1}. On the other hand for every $n \geq 1$ we have
\[\dist(q_n\vartheta,\Z)=|q_n \vartheta - p_n| <\frac{1}{q_{n+1}}=\frac{1}{\lceil q_n^{\gamma-1} \rceil q_n+q_{n-1}}<\frac{1}{q_n^{\gamma}}\]
and this proves \eqref{eq:cf2}. In the case $\gamma=1$ we have $q_n =q_{n-1}+q_{n-2} \leq 2q_{n-1}$ for every $n \geq 1$ which implies the boundedness of the sequence $(q_{n+1}/q_n)_{n=1}^\infty = (n_{\ell+1}/n_\ell)_{\ell=1}^\infty$. The proof of the lemma is complete.
\end{proof}

\begin{proof}[Proof of Theorem \ref{th:pyjama}]
Given $\alpha \in [\frac{1}{3},\frac{1}{2})$ define $\gamma:=\frac{\alpha}{1-2\alpha}\geq1$, let $\theta \in \R$ be as given by Lemma \ref{le:dio} and let $\A=\{A_0, A_1\}$ and $(a_n)$ be as given in the statement of the theorem. We observe that if $\alpha=\frac{1}{3}$ then $\gamma=1$. In order to establish the bound $a_n=O(n^\alpha)$ we begin with the following claim: there exists a constant $K_1>0$ such that for all integers $m\geq 1$ and $n_1,\ldots,n_m\geq 1$,
\[\left\|\left(A_0A_1^{n_1}A_0\right)\left(A_0A_1^{n_2}A_0\right) \cdots \left(A_0A^{n_m}_1A_0\right)\right\| \leq 1+ \left(K_1\sum_{i=1}^m n_i\right)^\alpha.\] 

A straightforward induction demonstrates that for all $n \geq 0$
\[A_1^n=\begin{pmatrix}1&\sin n\theta&\cos n\theta -1\\ 0&\cos n\theta & -\sin n\theta \\ 0&\sin n\theta &\cos n\theta\end{pmatrix}\]
and therefore
\begin{equation}\label{eq:that-one}A_0A_1^nA_0 = \begin{pmatrix}1&\sin n\theta&0\\ 0&\cos n\theta &0 \\ 0&0 &0\end{pmatrix}. \end{equation}
Fix integers $m \geq 1$ and $n_1,\ldots,n_m \geq 1$. For each $k \geq 1$ we may write
\[\left(A_0A_1^{n_1}A_0\right)\left(A_0A_1^{n_2}A_0\right) \cdots \left(A_0A_1^{n_k}A_0\right) =  \begin{pmatrix}1&w_k&0\\ 0&\prod_{i=1}^k \cos n_i\theta &0 \\ 0&0 &0\end{pmatrix} \]
where the real numbers $w_1,\ldots, w_m$ satisfy $w_1=\sin n_1\theta$ and $w_{k}=\sin n_{k} \theta + w_{k-1} \cos n_{k}\theta$ for every $k=2,\ldots,m$. Define further real numbers $z_1,\ldots,z_m$ inductively by $z_1:=|w_1|$ and $z_{k}=|\sin n_{k}\theta|+ z_{k-1} |\cos n_{k}\theta|$ for every $k=2,\ldots,m$. A trivial induction demonstrates that $|w_i|\leq z_i$ for every $i=1,\ldots,m$. Define additionally $z_0:=0$. By Lemma \ref{le:tricalc} we have
\[\left\|\left(A_0A_1^{n_1}A_0\right)\left(A_0A_1^{n_2}A_0\right) \cdots \left(A_0A_1^{n_m}A_0\right)\right\| \leq 1+|w_m|\leq 1+z_m\]
and so the claim will follow if we establish that $z_m \leq (K\sum_{i=1}^m n_i)^\alpha$ for some constant $K>0$ depending only on $\theta$.

By Lemma \ref{le:dio} there exists $\kappa>0$ such that $|\sin n\theta|\geq \kappa n^{-\gamma}$ for every integer $n \geq 1$. In particular for every $n \geq 1$ we have $n\theta \notin \R\setminus \pi\Z$, so applying Lemma \ref{le:bend} with $\beta:=1/\gamma \in (0,1]$ and $\delta:=1$ there exists $C>0$ such that
\begin{equation}\label{eq:pee}|\sin n\theta| + p|\cos n\theta| \leq \left(p^{2+\frac{1}{\gamma}} + \frac{C}{|\sin n\theta|^{\frac{1}{\gamma}}}\right)^{\frac{1}{2+\frac{1}{\gamma}}}= \left(p^{\frac{1}{\alpha}} + \frac{C}{|\sin n\theta|^{\frac{1}{\gamma}}}\right)^\alpha \end{equation}
for every positive integer $n$ and real number $p\geq 1$. Define $K:=\max\{2^{1/\alpha}, C\kappa^{-1/\gamma}\}$. If  $z_{m-1}<1$ then clearly
\[z_m=|\sin n_m\theta|+ z_{m-1}|\cos n_m\theta|< 2\leq K^\alpha  \leq\left(K\sum_{i=1}^m n_i\right)^\alpha\]
as required. Otherwise there exists an integer $r$ in the range $1 \leq r < m$ such that $z_{r-1}<1$ and $z_{i} \geq 1$ for all $i=r,\ldots,m-1$. For all $i=r+1,\ldots,m$ we have
\[z_{i}=|\sin n_{i}\theta|+ z_{i-1} |\cos n_{i}\theta| \leq \left(z_{i-1}^{\frac{1}{\alpha}} +\frac{C}{|\sin n_i \theta|^{\frac{1}{\gamma}}}\right)^{\alpha} \leq \left(z_{i-1}^{\frac{1}{\alpha}} +C\kappa^{-\frac{1}{\gamma}} n_i\right)^{\alpha}\]
where the inequality $z_{i-1} \geq 1$ permits the application of \eqref{eq:pee}. Thus for every $i=r+1,\ldots,m$ we have
\[z_i^{\frac{1}{\alpha}} \leq z_{i-1}^{\frac{1}{\alpha}} + Kn_i\]
so that by a trivial induction
\[z_m^{\frac{1}{\alpha}} \leq z_r^{\frac{1}{\alpha}} + K\sum_{i=r+1}^m n_i.\]
Since $0 \leq z_{r-1}<1$ we have
\[z_r=|\sin n_r\theta|+z_{r-1}|\cos n_r\theta|\leq |\sin n_r\theta| +|\cos n_r\theta|<2,\]
whence
\[z_m^{\frac{1}{\alpha}}  \leq z_r^{\frac{1}{\alpha}} + K\sum_{i=r+1}^m n_i< 2^{\frac{1}{\alpha}} + K\sum_{i=r+1}^m n_i \leq K\sum_{i=r}^m n_i \leq K\sum_{i=1}^m n_i.\]
We conclude that in all cases
\[\left\|\left(A_0A_1^{n_1}A_0\right)\left(A_0A_1^{n_2}A_0\right) \cdots \left(A_0A_1^{n_m}A_0\right)\right\| \leq 1+z_m \leq 1+\left(K\sum_{i=1}^m n_i\right)^\alpha\]
and the claim is proved. 

We may now prove the upper bound in the statement of Theorem \ref{th:pyjama}. We observe that $A_0^2=A_0$ and therefore $A_0^n=A_0$ for every $n \geq 1$. It follows in particular that $\|A_0^n\|=1$ for all $n \geq 0$. On the other hand since $A_1$ is diagonalisable over $\mathbb{C}$ the sequence $(A_1^n)_{n=1}^\infty$ is bounded and we may choose $K_2 \geq 1$ such that $\|A_1^n\| \leq K_2$ for every $n \geq 0$. Now let $M \in M_3(\R)$ be a product of the matrices $A_0$, $A_1$ of length $n \geq 1$. If $M$ includes a factor of the form $A_0 A_1^k A_0$ for some integer $k \geq 1$ then it may be written in the form
\[M= A_1^{k_1} A_0^{t_1} A_1^{r_1}A_0^{t_2} A_1^{r_2} \cdots A_0^{t_m}A_1^{r_m}A_0^{t_{m+1}} A_1^{k_2} \]
where $k_1,k_2\geq 0$,  where $r_i, t_i \geq 1$ for all $i$, and where $k_1+k_2+\sum_{i=1}^{m} r_i + \sum_{i=1}^{m+1} t_i=n$. Taking advantage of the identity $A_0=A_0^k$ for all $k \geq 1$ we have
\[M= A_1^{k_1} (A_0 A_1^{r_1}A_0)(A_0A_1^{r_2}A_0) \cdots (A_0 A_1^{r_m}A_0) A_1^{k_2}\]
and therefore
\begin{align*}\|M\| &\leq K_2^2 \left\| (A_0 A_1^{r_1}A_0)(A_0A_1^{r_2}A_0) \cdots (A_0 A_1^{r_m}A_0)\right\|\\
& \leq K_2^2\left(1+\left(K_1 \sum_{i=1}^m r_i\right)^\alpha\right) <2K_1^\alpha K_2^2 n^\alpha\end{align*}
using the preceding claim. 
If on the other hand $M$ does \emph{not} include a factor of the form $A_0 A_1^k A_0$ for any integer $k \geq 1$ then it can only have the form $M=A_1^{k_1} A_0^{k_2} A_1^{k_3}$ where $k_1, k_2, k_3$ are non-negative integers, and in this case we have
\[\|M\| \leq \|A_1^{k_1}\|\cdot \|A_0\|\cdot \|A_1^{k_3}\|\leq K_2^2< K_1^\alpha K_2^2 < 2K_1^\alpha K_2^2n^\alpha.\]
In all cases we have $\|M\| \leq 2K_1^\alpha K_2^2 n^\alpha$ and the upper bound is proved.

We now turn our attention to the lower bound. For each $n \geq 1$, a simple induction on $m\geq1$ using the formula \eqref{eq:that-one} demonstrates that
\[\left(A_0A_1^nA_0\right)^m=\begin{pmatrix} 1& \sum_{k=0}^{m-1} \sin n\theta (\cos n \theta)^k&0\\ 0&(\cos n\theta)^m&0\\ 0&0&0\end{pmatrix}=\begin{pmatrix} 1&(\sin n\theta) \frac{1- (\cos n \theta)^m}{1-\cos n\theta}&0\\ 0&(\cos n\theta)^m&0\\ 0&0&0\end{pmatrix}.\]
We have
\[|\cos n\theta|^m =\left|1-|\sin n\theta|^2\right|^{\frac{m}{2}} \leq \exp\left(-\frac{m|\sin n\theta|^2}{2}\right)\leq\exp\left(-\frac{\kappa^2 m}{2n^{2\gamma}}\right)  \]
using the elementary inequality $|1-x|\leq e^{-x}$ which is valid for all $x \in [0,1]$ followed by the inequality $|\sin n\theta| \geq \kappa n^{-\gamma}$ noted previously. Using the reverse triangle inequality we deduce that
\[|1-(\cos n\theta)^m| \geq 1-|\cos n\theta|^m \geq 1-\exp\left(-\frac{\kappa^2m}{2n^{2\gamma}}\right)>0\]
for all $n,m \geq 1$. If $\cos n\theta$ is positive, then we have additionally
\[\left|\frac{1}{1-\cos n\theta}\right| =\frac{1}{1-|\cos n\theta|} =\frac{1}{1-\sqrt{1-|\sin n\theta|^2}} \geq \frac{1}{|\sin n\theta|^2} \]
where we have used the inequality $\sqrt{1-x} \geq 1-x$ for $x \in [0,1]$. Combining the preceding estimates yields
\[\left\|(A_0A_1^nA_0)^m\right\| \geq \left|\frac{(\sin n\theta)(1-(\cos n\theta)^m)}{1-\cos n\theta}\right|\geq \frac{1-\exp\left(-\frac{\kappa^2m}{2n^{2\gamma}}\right)}{|\sin n\theta|}\]
for every pair of integers $n,m \geq 1$ such that $\cos n\theta>0$.

By Lemma \ref{le:dio} we may choose a real number $K>0$ and a sequence of integers $(q_n)_{n=1}^\infty$ increasing to infinity such that $|\sin q_n\theta|\leq Kq_n^{-\gamma}$ for every $n \geq 1$. If $\alpha=1/3$ then $\gamma=1$ and therefore in that case the sequence $(q_n)$ may be chosen so that the sequence of ratios $(q_{n+1}/q_n)$ is bounded. For all $n \geq 1$ we have $|\sin 2q_n\theta| =|2\sin q_n\theta \cos q_n\theta|\leq 2Kq_n^{-\gamma}$. If $n$ is large enough that $2K^2q_n^{-2\gamma}<1$ then we also have $2|\sin q_n\theta|^2<1$ and therefore $\cos 2q_n\theta = 1-2\sin^2 q_n\theta>0$. For all $n$ large enough that $2K^2q_n^{-2\gamma}<1$ we therefore have
\begin{align*}\left\|\left(A_0A_1^{2q_n}A_0\right)^{\left\lceil q_n^{2\gamma}\right\rceil}\right\| &\geq \frac{1-e^{-\frac{\kappa^2 \left\lceil q_n^{2\gamma}\right\rceil}{2\left(2q_n\right)^{2\gamma}}}}{\left|\sin 2q_n\theta\right|}\\
&\geq \frac{\left(1-e^{-\frac{\kappa^2}{2^{2\gamma+1}}}\right) q_n^\gamma}{2K}\\
&= \frac{\left(1-e^{-\frac{\kappa^2}{2^{2\gamma+1}}}\right) q_n^{(2\gamma+1)\alpha}}{2K}\geq \frac{1-e^{-\frac{\kappa^2}{2^{2\gamma+1}}}}{2^{6\alpha\gamma+1} K}\left(\lceil q_n^{2\gamma}\rceil(2q_n+2)  \right)^{\alpha}.
\end{align*}
In particular for every large enough $n \geq 1$ there exists a product of the matrices $A_0, A_1$ with length $\lceil q_n^{2\gamma}\rceil(2q_n+2)$ and norm greater than or equal to a constant times $(\lceil q_n^{2\gamma}\rceil(2q_n+2))^\alpha$. In the case $\alpha=1/3$, $\gamma=1$ the sequence of ratios $(q_{n+1}/q_n)$ is bounded above and therefore the sequence of ratios $(\lceil q_{n+1}^{2\gamma}\rceil(2q_{n+1}+2))/(\lceil q_n^{2\gamma}\rceil(2q_n+2))$ is also bounded above. We have established the claimed lower bound, and the theorem is proved.
\end{proof}

\section{Conclusions}

In the case of a single matrix (or a discrete linear system without switching) the class of marginal instability rate sequences is extremely constrained: every marginal instability rate sequence in $d$ dimensions is asymptotic to a monomial with a non-negative integer exponent strictly smaller than $d$, and no other cases are possible. The situation for sets with more than one matrix (or for discrete linear switched systems with arbitrary switching) seems to be very much the reverse: we have found no completely general constraints on the scope of behaviour of marginal instability rate sequences other than that the sequence must be bounded above by a constant times $n^{d-1}$, must bounded below by a constant and cannot fall or rise by an arbitrarily large proportion in a single step (that is, the ratio $a_{n}/a_{n+1}$ must be bounded away from both zero and infinity). We have also found abundant ways of constructing new marginal instability rate functions from old ones. These considerations taken together suggest that the class of marginal instability rate sequences is rather large and that relatively few properties hold universally within this class of sequences. We nonetheless ask the following question:

\begin{question}\label{qu:est} Are all marginal instability rate sequences weakly increasing, and are they all weakly upper regularly varying?\end{question}

We have considered the problem of modifying the argument used in Theorem \ref{th:no-sudden-moves} so as to attempt to show that every marginal instability rate sequence is weakly increasing. It is not difficult to establish an induction step along the same lines which shows that in the block upper triangular representation
\[\{A_i\colon i \in \I\}=\left\{\begin{pmatrix} B_i & D_i \\ 0&C_i\end{pmatrix} \colon i \in \I\right\}\]
the weakly increasing property holds for $\{A_i\}$ if it holds for either one of $\{B_i\}$ and $\{C_i\}$ while the other has joint spectral radius strictly less than $1$. The problematic case occurs when $\varrho(\{B_i\})=\varrho(\{C_i\})=1$. Here  it seems to become unavoidably necessary to consider a product of the form
\[\begin{pmatrix} B & D \\ 0&C\end{pmatrix}\begin{pmatrix} B' & D' \\ 0&C'\end{pmatrix}\]
in which $\|D\|$ and $\|D'\|$ are of similar magnitude. In this circumstance it is not clear how to avoid additive cancellation in the upper-right block of the product other than by imposing very strong \emph{a priori} hypotheses such as non-negativity of the blocks. (The latter hypothesis anyway seems likely  to be so strong as to restrict the marginal instability rate sequence to a very simple form such as $a_n \simeq n^k$ for integer $k$.) This leads us to suggest that the answer to the first part of Question \ref{qu:est} above might be negative, although we are not able to prove this.

We ask a further question which concerns the set of allowed exponents in marginal instability rate sequences $(a_n)$ with the property $a_n \simeq n^\alpha$.  Let $\mathscr{E} \subset [0,\infty)$ denote the set of all real numbers $\alpha \geq 0$ such that there exist $d\geq 1$ and a finite set $\A \subset M_d(\R)$ whose marginal instability rate sequence satisfies $a_n \simeq n^\alpha$. By Theorem \ref{th:BIG-TWOS}\eqref{it:you-know-what-will-get-you-you-know-where} we may without loss of generality assume that $\A$ has cardinality $2$. It follows from Proposition \ref{pr:stick-with-the-prod} that the set $\mathscr{E}$ is an additive semigroup, and by Corollary \ref{co:harvey} and Corollary \ref{co:lon} respectively we have $\frac{1}{3} \in \mathscr{E}$ and $0 \in \mathscr{E}$. We ask:
\begin{question} Do we have $\mathscr{E}=[0,\infty)$?\end{question}

\section{Acknowledgements} 

Versions of Theorems \ref{th:first-obe}, \ref{th:BIG-TWOS}, \ref{th:no-sudden-moves} and \ref{th:pyjama} previously appeared in the first named author's PhD thesis \cite{Va22}. J. Varney was supported by EPRSC Doctoral Training Partnership grant EP/R513350/1. I.D. Morris was partially supported by Leverhulme Trust Research Project Grant RPG-2016-194.

\bibliographystyle{acm}
\bibliography{vsb}

\end{document}